\documentclass[11pt]{article}

\usepackage{fancyhdr}
\usepackage{amsfonts}
\usepackage{amsmath}
\usepackage{amssymb}
\usepackage{amsthm}
\usepackage{tikz}

\usepackage{amscd}
\usepackage{amstext}

\usepackage{graphics}
\usepackage{graphicx}

\newlength{\fixboxwidth}
\newcommand{\re}{\mathbb{R}}\newcommand{\N}{\mathbb{N}}
\newcommand{\zz}{\mathbb{Z}}
\newcommand{\tor}{\mathbb{T}}

\newcommand{\Z}{{\zz}^d}

\newcommand{\R}{{\re}^d}

\newcommand{\cs}{{\mathcal S}}

\newcommand{\cl}{{\mathcal L}}

\newcommand{\ce}{{\mathcal E}}

\newcommand{\cf}{{\mathcal F}}
\newcommand{\cfi}{{\cf}^{-1}}

\newcommand{\cz}{{\mathcal Z}}

\newcommand{\supp}{{\rm supp \, }}

\newcommand{\ls}{\lesssim}
\newcommand{\mix}{{\rm mix}}

\newcommand{\bproof}{\begin{proof}}
\newcommand{\eproof}{\end{proof}}

\newcommand{\be}{\begin{equation}}
\newcommand{\ee}{\end{equation}}
\newcommand{\beq}{\begin{eqnarray}}
\newcommand{\beqq}{\begin{eqnarray*}}
\newcommand{\eeq}{\end{eqnarray}}
\newcommand{\eeqq}{\end{eqnarray*}}

\numberwithin{equation}{section}

\newtheorem{theorem}{Theorem}[section]
\newtheorem{definition}[theorem]{Definition}
\newtheorem{corollary}[theorem]{Corollary}
\newtheorem{lemma}[theorem]{Lemma}
\newtheorem{proposition}[theorem]{Proposition}
\newtheorem{remark}[theorem]{Remark}

\begin{document}


\title{Weyl Numbers of Embeddings of Tensor Product Besov Spaces}

\author{Van Kien Nguyen \& Winfried Sickel\\
Friedrich-Schiller-University Jena, Ernst-Abbe-Platz 2, 07737 Jena, Germany\\
kien.nguyen@uni-jena.de \quad \& \quad  winfried.sickel@uni-jena.de
}


\date{\today}

\maketitle

\begin{abstract}
In this paper we investigate the asymptotic behaviour of Weyl numbers of embeddings of
tensor product Besov spaces into Lebesgue spaces. These results will be compared with the known
behaviour of entropy numbers. 
\end{abstract}


\section{Introduction}\label{intro}
 

Weyl numbers have been introduced by Pietsch \cite{Pi1}.
They are relatives of approximation numbers.
Recall, the $n$-th approximation number of the linear operator $T \in \mathcal L(X,Y)$ is defined to be 
\be\label{app}
 a_n(T):=\inf\{\|T-A\|: \ A\in \mathcal L(X,Y),\ \ \text{rank} (A)<n\}\, , \qquad n \in \N\, . 
 \ee
Here $X$ and $Y$ are quasi-Banach spaces.
Now we are in position to define Weyl numbers.
\\  
The $n$-th Weyl number of the linear operator $T \in \mathcal L(X,Y)$ is given by 
$$ x_n(T):=\sup\{a_n(TA):\ A\in \mathcal L(\ell_2,X),\ \|A\|\leq 1\}\, , \qquad n \in \N\, .$$
Approximation and Weyl numbers belong to the family of $s$-numbers, see Section \ref{basic} for more details.
The particular interest in Weyl numbers stems from the fact that 
they are the smallest known $s$-number satisfying the famous Weyl-type inequalities, i.e.,
\be\label{ws-01}
\Big( \prod_{k=1}^{2n-1} |\lambda_k (T)|\Big)^{1/(2n-1)} \le \sqrt{2e} \, \Big( \prod_{k=1}^{n} x_k (T)\Big)^{1/n}
\ee
holds for all $n \in \N$, in particular, 
\[
|\lambda_{2n-1} (T)| \le \sqrt{2e} \, \Big( \prod_{k=1}^{n} x_k (T)\Big)^{1/n} \, ,
\]
see Pietsch \cite{Pi1} and Carl, Hinrichs \cite{CH}.
Here $T: \, X \to X $ is a compact linear operator
in a Banach space $X$ and 
$(\lambda_n(T))_{n=1}^\infty$ denotes the sequence of all non-zero
eigenvalues of $T$, repeated according to algebraic multiplicity
and ordered such that 
\[
|\lambda_1(T)| \ge |\lambda_2(T)| \ge \, \ldots \, \ge 0\, .
\]
Also as a consequence of \eqref{ws-01} one obtains, for all $p\in (0,\infty)$, the existence
of a constant $c_p$ (independent of $T$) s.t.
\[
\Big(\sum_{n=1}^\infty |\lambda_n (T)|^p\Big)^{1/p}
\le c_p \, \Big(\sum_{n=1}^\infty x^p_n (T)\Big)^{1/p}\, .
\]
Hence, Weyl numbers may be seen as an appropriate tool to control the eigenvalues of $T$.
Many times operators of interest can be written as a composition of an identity between appropriate function spaces 
and a further bounded operator, see, e.g., the monographs of K\"onig \cite{Koe} and of Edmunds, Triebel \cite{ET}.
This motivates the study of Weyl numbers of identity operators. Pietsch \cite{Pi4}, 
Lubitz \cite{Claus}, K\"onig \cite{Koe} and Caetano \cite{Cae1,Cae2,Caed} 
studied the Weyl numbers of $id: B^t_{p_1,q_1}((0,1)^d) \to L_{p_2}((0,1)^d)$, where
$B^t_{p_1,q_1}((0,1)^d)$ denotes the isotropic Besov spaces.
Zhang, Fang and Huang \cite{Fang1} and Gasiorowska and Skrzypczak \cite{GS} investigated the case of embeddings of weighted Besov spaces, defined on $\R$, into Lebesgue spaces.
Here we are interested in the investigation of the asymptotic behaviour of 
Weyl numbers of the identity 
\[
id: S^t_{p_1,p_1}B((0,1)^d) \to L_{p_2}((0,1)^d) \, ,
\]
where $S^t_{p_1,p_1}B((0,1)^d)$ denotes a d-fold tensor product of univariate Besov spaces $ B^t_{p_1,p_1}(0,1)$.
This notation is chosen in accordance with the fact that
$S^t_{p_1,p_1}B((0,1)^d)$ can be interpreted as a special case of the scale of Besov spaces 
of dominating mixed smoothness, see Section \ref{domino}. 
\\
The behaviour of 
$x_n (id:\, S^t_{p_1,p_1}B((0,1)^d)\to L_{p_2}((0,1)^d))$, $1 <p_2 <\infty$,
will be discussed in Subsection \ref{paley}. Here, up to some limiting situations, we have the complete 
picture, i.e., we know the exact asymptotic behaviour of the Weyl numbers.
For the extreme cases $p_2 = \infty$ and $p_2 = 1$, 
see Subsection \ref{extrem} and Subsection \ref{extrem1},
we are also able to describe the behaviour in almost all situations.  
In Subsection \ref{extrem2} we discuss the behaviour of 
$x_n (id:\, S^t_{p_1,p_1}B((0,1)^d)\to \cz^s_\mix((0,1)^d))$, where $\cz^s_\mix ((0,1)^d)$ denotes a space of 
H\"older-Zygmund type. Finally, in Subsection \ref{entropy}, we compare the behaviour of  Weyl numbers 
with that one of entropy numbers.
\\
Summarizing we present an almost complete picture of the behaviour of $x_n (id:\, S^t_{p_1,p_1}B((0,1)^d)\to L_{p_2}((0,1)^d))$, 
$0 < p_1 \le \infty$, $1\le p_2 \le \infty$, $t > \max(0, \frac{1}{p_1}-\frac{1}{p_2})$.
This is a little bit surprising since the corresponding results for approximation numbers 
are much less complete, see, e.g., \cite{Baz7}. 
Let us mention in this context that Weyl numbers have some specific properties not shared 
by approximation numbers like interpolation properties, see Theorem \ref{inter}, or the inequality \eqref{kuehn}.
\\
The paper is organized as follows. Our main results are discussed in Section \ref{main1}. 
In Section \ref{basic} we recall the definition of $s$-numbers and discuss some further  properties of Weyl numbers.
Section \ref{domino} is devoted to the function spaces under consideration.
In Section \ref{sequence} we investigate the Weyl numbers of embeddings of certain sequence spaces
associated to tensor product Besov spaces and spaces of dominating mixed smoothness.
This will be followed by Section \ref{proof}, where, beside others, Theorem \ref{main} (our main result) will be proved.
In Appendix A we recall the behaviour of the Weyl numbers of embeddings $id_{p_1,p_2}^m:~ \ell_{p_1}^m \to \ell_{p_2}^m$.
Finally, in Appendix B, a few more facts about the Lizorkin-Triebel spaces of dominating mixed smoothness 
$S^t_{p,q}F(\R)$, $S^t_{p,q}F((0,1)^d)$ and the Besov spaces of dominating mixed smoothness $S^t_{p,q}B(\R)$, $S^t_{p,q}B((0,1)^d)$
are collected.


\subsection*{Notation}


As usual, $\N$ denotes the natural numbers, $\N_0 := \N \cup \{0\}$,
$\zz$ the integers and
$\re$ the real numbers. For a real number $a$ we put $a_+ := \max(a,0)$.
By $[a]$ we denote the integer part of $a$.
If $\bar{j} \in \N_0^d$, i.e., if $\bar{j}=(j_1, \ldots \, , j_d)$, $j_\ell \in \N_0$, $\ell=1, \ldots \, , d$, then we put
\[
|\bar{j}|_1 := j_1 + \ldots \, + j_d\, .
\]
By $\Omega$ we denote the unit cube in $\R$, i.e., $\Omega:= (0,1)^d$.
If $X$ and $Y$ are two quasi-Banach spaces, then the symbol $X \hookrightarrow Y$ indicates that the embedding is continuous. 
As usual, the symbol $c $ denotes positive constants 
which depend only on the fixed parameters $t,p,q$ and probably on auxiliary functions, unless otherwise stated; its value may vary from line to line.
Sometimes we will use the symbols ``$ \lesssim $'' 
and ``$ \gtrsim $'' instead of ``$ \le $'' and ``$ \ge $'', respectively. The meaning of $A \lesssim B$ is given by: there exists a constant $c>0$ such that
 $A \le c \,B$. Similarly $\gtrsim$ is defined. The symbol 
$A \asymp B$ will be used as an abbreviation of
$A \lesssim B \lesssim A$.
For a discrete set $\nabla$ the symbol $|\nabla|$ 
denotes the cardinality of this set.
Finally, the symbols ${id}, id^*$ will be used for identity operators, $id^*$ mainly in connection with sequence spaces.
The symbol $id_{p_1,p_2}^m$ refers to the identity 
\be\label{idlp}
id_{p_1,p_2}^m:~ \ell_{p_1}^m \to \ell_{p_2}^m\, .
\ee
Tensor products of Besov and Sobolev spaces are investigated 
in \cite{Spreng}, \cite{SUt}, \cite{SUspline} and Hansen \cite{Hansen}.
General information about Besov and Lizorkin-Triebel spaces of dominating mixed smoothness can be found, e.g., in 
\cite{Am,Baz1,Baz2,Baz3,Sc2,ST,Vybiral} ($S^t_{p,q}B(\R)$, $S^{t}_{p,q}F(\R)$).
The (Fourier analytic) definitions of these spaces are reviewed in the Appendix B.
The reader, who is interested in more elementary descriptions of these spaces, e.g., by means of differences, is 
referred to \cite{Am,ST} and \cite{U1}.


\section{Some preparations}


As a preparation for the main results we recall under which conditions the identity $S^t_{p_1,p_1}B((0,1)^d) \hookrightarrow L_{p_2}((0,1)^d)$ is 
compact, see Vybiral \cite[Thm.~3.17]{Vybiral}.

\begin{proposition}\label{emb}
The following assertions are equivalent:
\\
{\rm (i)} The embedding $S^t_{p_1,p_1}B((0,1)^d) \hookrightarrow L_{p_2}((0,1)^d)$
is compact;
\\
{\rm (ii)} We have
\be\label{ws-02}
t > \max \Big(0, \frac{1}{p_1}- \frac{1}{p_2} \Big) \, .
\ee
 \end{proposition}

Since we are exclusively interested in compact embeddings the restriction \eqref{ws-02} will be always present.
\\
Also for later use, we recall the Weyl numbers of the embedding $id: B^t_{p_1,q_1}(0,1) \to L_{p_2}(0,1)$.
Let $0 < p_1,q_1 \le \infty$, $1 \le p_2 \le \infty$ and $t> \max (0, 1/p_1-1/p_2)$. Then, in all cases listed in Prop. \ref{isotrop},  we have
\be\label{ws-03}
x_n (id)= x_n (id: ~B^t_{p_1,q_1}(0,1) \to L_{p_2}(0,1)) \asymp n^{-\alpha}\, , \qquad n \in \N\, .
\ee
Here the value of $\alpha = \alpha (t,p_1,p_2)$ is given in the following proposition.

\begin{proposition}\label{isotrop}
 The value of $\alpha$ in \eqref{ws-03} is given by
\begin{itemize}
\item[$I$]\qquad{\makebox[3.5cm][l]{$\alpha= t$} if\ \ $ p_1,p_2 \le 2$};
\item[$II$]\qquad{\makebox[3.5cm][l]{$\alpha= t + \frac{1}{p_2} - \frac 12$} if\ \ $p_1 \le 2 \le p_2 $};
\item[$III$]\qquad{\makebox[3.5cm][l]{$\alpha=t + \frac{1}{p_2} - \frac{1}{p_1} $} if\ \ $ 2 \le p_1 \le p_2$};
\item[$IV^*$]\qquad{\makebox[3.5cm][l]{$\alpha= t + \frac{1}{p_2} - \frac{1}{p_1} $} if\ \ $2 \le p_2 < p_1 \text{ and } t> \frac{1/p_2 - 1/p_1}{p_1/2 - 1} $};
\item[$IV_*$]\qquad{\makebox[3.5cm][l]{$\alpha= \frac{tp_1}{2} $} if\ \ $2 \le p_2 < p_1 <\infty\text{ and } t < \frac{1/p_2 - 1/p_1}{p_1/2 - 1}$};
\item[$V^*$]\qquad{\makebox[3.5cm][l]{$\alpha= t - \frac{1}{p_1} + \frac 12$} if\ \ $p_2 \le 2 < p_1 \text{ and } t>\frac{1}{p_1}$};
\item[$V_*$]\qquad{\makebox[3.5cm][l]{$\alpha= \frac{tp_1}{2}$} if\ \ $ p_2 \le 2 < p_1<\infty \text{ and } t< \frac{1}{p_1}$}.
\end{itemize}
 \end{proposition}

\begin{minipage}[t]{7cm}
$$
\begin{tikzpicture}
\fill (0,0) circle (1.5pt);
\draw[->, ](0,0) -- (5,0);
\draw[->, ] (0,0) -- (0,5);
\draw (2,0)-- (2,4);
\draw (0,0) -- (2,2);
\draw (0,2) -- (5,2);
\draw (0,4) -- (5,4);
\node[below] at (0,0) {$0$};
\node [below] at (2,0) {$\frac{1}{2}$};
\node [left] at (0,2) {$\frac{1}{2}$};
\node [left] at (0,4) {$ 1$};
\node [left] at (0,5) {$\frac{1}{p_2}$};
\node [below] at (4,0) {$1$};
\draw (4, -0.05) -- (4, 0.05);
\node [below] at (5,0) {$\frac{1}{p_1}$};
\node [] at (3,3) {$I$};
\node [] at (3,1) {$II$};
\node [] at (1,3) {$V$};
\node [right] at (1,1/2) {$III$};
\node [above] at (1/2,1) {$IV$};
\node [right] at (2,-1) {Figure 1};
\end{tikzpicture}
$$
\end{minipage}
\hfill
\begin{minipage}[t]{6cm}
{~}\\ 
The above results indicate a decomposition of the $(1/p_1,1/p_2)$-plane into five parts.
In regions $IV$ and $V$ we have a further splitting
into the cases of small ($IV_*$, $V_*$) and large smoothness ($IV^*$, $V^*$).
Proposition \ref{isotrop} has been proved by Pietsch \cite{Pi4} and  Lubitz \cite[Satz 4.13]{Claus} in case $1 \le p_1, q_1,p_2 \le \infty$, 
we refer also to K\"onig \cite{Koe} in this context. 
\end{minipage}

{~ }\\
\noindent
The proof of the general case, i.e., also for values of  $p_1,q_1$  less than $1$,  
can be found in the thesis of Caetano \cite{Caed}, see also \cite{Cae2}.
Obviously we do not have a dependence on the fine-index $q_1$. This will be different in the dominating mixed case with $d>1$.
The behaviour of the Weyl numbers in the limiting situations 
$t = \frac{1/p_2 - 1/p_1}{p_1/2 - 1}$ (see $IV_*, IV^*$) and $t=1/p_1$ (see $V_*,V^*$) is open, in particular it seems to be unknown 
whether the behaviour is still polynomial in $n$.


\section{The main results}\label{main1}


It seems to be appropriate to split our considerations into the three cases: 
(i) $1 < p_2 < \infty$, (ii) $p_2 = \infty$ and  (iii) $ p_2 =1. $


\subsection{The Littlewood-Paley case}
\label{paley}

Littlewood-Paley analysis is one of the main tools to understand the behaviour of the Weyl numbers 
if $1 < p_2 < \infty$
(i.e., the target space $L_{p_2}$ allows a Littlewood-Paley-type decomposition).
The cases $p_2 = 1$ and $p_2 = \infty$ require different techniques and will be treated in the next subsections.
As in the isotropic case the results suggest to work with the same decomposition 
of the $(1/p_1,1/p_2)$-plane as in Proposition \ref{isotrop}. So, the symbols $I,II, \ldots$, used below, 
have the same meaning as in Figure 1 (and therefore as in Prop. \ref{isotrop}). 
In addition the regions $I^*$ and $I_*$ are given by $p_1,p_2 \le 2$.
\\
Let  $0 < p_1 \le \infty$, $1< p_2 <\infty$ and $t> \max (0, 1/p_1-1/p_2)$. Then in all cases, listed in Theorem \ref{main}, we have
\be\label{ws-05}
x_n (id)= x_n (id: S^t_{p_1,p_1}B((0,1)^d) \to L_{p_2}((0,1)^d)) \asymp n^{-\alpha}\, 
(\log n)^{(d-1)\, \beta}\, , \qquad n\geq 2\, . 
\ee
The values of $\alpha = \alpha (t,p_1,p_2)$ and $\beta = \beta (t,p_1,p_2)$ are be given in the following theorem.

\begin{theorem}\label{main}
The values of $\alpha$ and $\beta$ in \eqref{ws-05} are given by
\begin{itemize}
\item[$I^*$] 
\qquad $\alpha = t $ and $\beta = t+\frac{1}{2}-\frac{1}{p_1}\quad \text{if }\quad t> \frac{1}{p_1}-\frac{1}{2}$;
\item[$I_*$]
\qquad $\alpha = t $ and $\beta = 0 \quad \text{if } \quad t<\frac{1}{p_1}-\frac{1}{2}$;
\item [$II$] 
\qquad $\alpha = t -\frac{1}{2}+\frac{1}{p_2}$ and $\beta = t+\frac{1}{p_2}-\frac{1}{p_1}$;
\item [$III$] 
\qquad $\alpha =t-\frac{1}{p_1}+\frac{1}{p_2} $ and $\beta = t+\frac{1}{p_2}-\frac{1}{p_1}$;
\item [$IV^*$] 
\qquad 
$\alpha =t-\frac{1}{p_1}+\frac{1}{p_2} $ and $\beta =t+\frac{1}{p_2}-\frac{1}{p_1} $ 
\quad if\quad $t>\frac{{1}/{p_2}-{1}/{p_1}}{{p_1}/{2}-1} $;
\item [$IV_*$]
\qquad
$\alpha =\frac{tp_1}{2} $ and $ \beta =t+\frac{1}{2}-\frac{1}{p_1}$ if 
\quad $t<\frac{{1}/{p_2}-{1}/{p_1}}{{p_1}/{2}-1}$;
\item [$V^*$] 
\qquad $\alpha =t-\frac{1}{p_1}+\frac{1}{2} $ and $\beta = t+\frac{1}{2}-\frac{1}{p_1}$ \quad
if \quad $t>\frac{1}{p_1}$;
\item [$V_*$]
\qquad 
$\alpha = \frac{tp_1}{2} $ and $\beta =t +\frac{1}{2}-\frac{1}{p_1}$ \quad if \quad $t<\frac{1}{p_1}$.
\end{itemize}
\end{theorem}

Thm. \ref{main} gives the final answer about the behaviour of the $x_n$
in almost all cases. It is interesting to notice that 
in regions $I,IV$ and $V$ we have a different behaviour for small smoothness compared with large smoothness.
Only in the resulting limiting cases we are not able to characterize the behaviour of the $x_n(id)$.
However, estimates from below and above also for these limiting situations will be given in Subsection \ref{sequence3}.
\\
In essence the proof is standard. Concerning the estimate from above
the first step consists in a reduction step.
By means of wavelet characterizations we switch from the consideration of 
$id: S^t_{p_1,p_1}B((0,1)^d) \to L_{p_2}((0,1)^d)$
to $id^*: s^{t, \Omega}_{p_1,p_1} b \to s^{0,\Omega}_{p_2,2}f$,
where $s^{t,\Omega}_{p,q}b$ and $s^{t,\Omega}_{p,q}f$ are appropriate sequence spaces.
Next, this identity is splitted into $id^* = \sum_{\mu=0}^\infty id_\mu^*$ (the $id_\mu^*$ are identities with respect to certain subspaces)
which results in an estimate of $x_n (id^*: s^{t, \Omega}_{p_1,p_1} b \to s^{0,\Omega}_{p_2,2}f)$
\[
x_n(id^*)\leq \sum_{\mu=0}^{J} x_{n_{\mu}}(id_{\mu}^*) + \sum_{\mu=J+1}^{L} x_{n_{\mu}}(id_{\mu}^*)
+ \sum_{\mu=L+1}^{\infty} \|id_{\mu}^*\|  
\]
where $n-1 = \sum_{\mu=0}^{L}(n_{\mu}-1)$, 
see \eqref{ws-16}. 
Till this point we would call the proof standard, compare, e.g., with Vybiral \cite{Vy}. But now the problem consists 
in choosing $J,L$ and $n_\mu$ in a way leading to the desired result.
This is the real problem which we solved in Subsection \ref{sequence3}.
In a further reduction step estimates of $x_{n_{\mu}}(id_{\mu}^*)$ are traced back to estimates of 
 $x_{n_\mu}(id_{p_1,p_2}^{D_\mu})$, see \eqref{idlp}. All what is needed about these number is collected in Appendix A.
Concerning the estimate from below one has to figure out appropriate subspaces of 
$S^t_{p_1,p_1}B((0,1)^d)$ ($s^{t,\Omega}_{p_1,p_1}b$). Then, also in this case, all can be reduced to the known estimates of 
 $x_{n}(id_{p_1,p_2}^{D_\mu})$.


\subsection{The extreme case $p_2 =\infty$}
\label{extrem}


Let us recall a result of Temlyakov \cite{Te93}, see also \cite{CKS}.

\begin{proposition}\label{satz 2}
Let $t>\frac{1}{2}$. Then we have 
\beqq
 x_{n} (id: \, S^{t}_{2,2}B ((0,1)^d) \to L_\infty ((0,1)^d)) \asymp
\frac{(\log n)^{(d-1)t}}{n^{t-\frac 12}}\, , \qquad n\geq 2\, .
\eeqq
\end{proposition}

\begin{remark}\label{banachhilbert}
 \rm
{\rm (i)} In the literature many times the notation 
$H^{t}_{\text{mix}} ((0,1)^d) $ and $MW^t_2 ((0,1)^d)$ are used instead of $S^{t}_{2,2}B ((0,1)^d)$.
\\
{\rm (ii)} 
In \cite{Te93} and \cite{CKS} the authors deal with approximation numbers
$a_{n} (id: \, S^{t}_{2,2}B ((0,1)^d) \to L_{\infty} ((0,1)^d))$, see \eqref{app}.
However, for Banach spaces $Y$ and Hilbert spaces $H$ we always have 
\beqq
x_n (T: H \to Y) = a_n (T: H \to Y)\, , 
\eeqq
see \cite[Prop.~2.4.20]{Pi3}.
\end{remark}

By using specific properties of Weyl numbers we will extent 
Prop. \ref{satz 2} to the following result.

\begin{theorem}\label{satz01}
Let $0<p_1\leq \infty$. Then we have
\begin{eqnarray}\label{ws-44}
x_{n} (id: \ S^{t}_{p_1,p_1}B((0,1)^d) &\to& L_\infty((0,1)^d)) 
\\
 & \asymp & 
\left\{
\begin{array}{lll}
\frac{(\log n)^{(d-1)(t+\frac 12 -\frac{1}{p_1} )}}{n^{t-\frac{1}{2}}} & \quad & \mbox{if}\quad 0 < p_1 \le 2\, ,\  t>\frac{1}{p_1}\, ;
\\
&& \\
\frac{(\log n)^{(d-1)(t + \frac 12 - \frac{1}{p_1})}}{n^{t-\frac{1}{p_1}}} & \quad & \mbox{if}\quad 2 < p_1 \le \infty\,  ,\  
t>\frac{1}{2}+\frac{1}{p_1}\, ;
\end{array}
\right.
\nonumber
\end{eqnarray}
for all $n \ge 2$.
\end{theorem}

\begin{remark}
 \rm
{\rm (i)}
Recall that $S^{t}_{p_1,p_1}B((0,1)^d)$ is compactly embedded into $ L_\infty((0,1)^d)$ if and only if $t> 1/{p_1}$, 
see Prop. \ref{emb}.
The cases $2<p_1\leq \infty$ and $t\in \big(\frac{1}{p_1},\frac{1}{p_1}+\frac{1}{2}\big]$ remain open.
\\ 
{\rm (ii)} Considering $p_2 \to \infty$ in parts II and III of Thm. \ref{main} then it turns out
that in \eqref{ws-44} there is an additional log factor, more exactly $(\log n)^{(d-1)/2}$. 
\\
{\rm (iii)} 
To prove Theorem \ref{satz01} 
we shall employ an inequality due to Pietsch \cite{Pi2}. 
For any linear operator $T$ we have
\be\label{kuehn}
n^{1/2}\, x_n (T) \le \pi_2 (T)\, , \qquad n \in \N\, , 
\ee
where $\pi_2 (T)$ refers to the 2-summing norm of $T$.
\\
{\rm (iv)} There is a small  number of cases,  where the exact order of 
$s_{n} (id: S^{t}_{p_1,q_1}B((0,1)^d) \to L_\infty((0,1)^d)) $, $p_1 \neq 2$, if $n$ tends to infinity, has been found.
Here $s_n$ stands for any $s$-number, see Sect. \ref{basic}.
Beside Prop. \ref{satz 2} we refer to Temlyakov \cite{Te96} where 
\[
d_{n} (id: S^{t}_{\infty, \infty}B(\tor^2) \to L_\infty(\tor^2)) \asymp n^{-t}\, (\log n)^{t+1} \, , \qquad t>0, 
\]
is proved for all $n \ge 2$. Here $d_n$ denotes the $n$-th Kolmogorov number and
$\tor^2$ refers to the two-dimensional periodic case.
For some partial results (i.e., with a gap between the estimates from above and below)
with respect to Kolmogorov numbers we refer to Romanyuk \cite{Rom3}.
\end{remark}


\subsection{The extreme case $p_2=1$}
\label{extrem1}


Let us recall a result obtained by Romanyuk \cite{Rom9} (again Romanyuk has dealt with
approximation numbers, but see Remark \ref{banachhilbert} for this).

\begin{proposition}\label{satz21}
Let $t>0$. Then we have
\beqq
 x_{n} (id: \, S^{t}_{2,2}B ((0,1)^d) \to L_1 ((0,1)^d)) \asymp
\frac{(\log n)^{(d-1)t}}{n^t}\, , \qquad n\geq 2\, .
\eeqq
\end{proposition}

By making use of the embedding $S^0_{1,2}F((0,1)^d) \hookrightarrow L_1 ((0,1)^d)$ 
we are able to extend Prop. \ref{satz21} to the following.

\begin{theorem}\label{extrem11}
Let $0<p_1\leq\infty $ and $t>(\frac{1}{p_1}-1)_+$. Then
\beqq x_n(id: S_{p_1,p_1}^tB((0,1)^d)& \to & L_1((0,1)^d ))
\\
&\asymp& 
\begin{cases}
 n^{- t} &  \text{if}\quad   p_1< 2\, , \ \ t< \frac{1}{p_1}-\frac{1}{2},\\
n^{-t}(\log n)^{(d-1)(t+\frac{1}{2}-\frac{1}{p_1})} & \text{if}\quad  p_1\leq 2\, , \ \ t> \frac{1}{p_1}-\frac{1}{2},\
\\
n^{-t+\frac{1}{p_1}-\frac{1}{2}}(\log n)^{(d-1)(t+\frac{1}{2}-\frac{1}{p_1})} &  
\text{if}\quad 2<p_1\leq \infty \, , \  \  t>\frac{1}{p_1}, 
\\
n^{-\frac{tp_1}{2}}(\log n)^{(d-1)(t+\frac{1}{2}-\frac{1}{p_1})} &  
\text{if}\quad 2< p_1<\infty\, , \  \  t<\frac{1}{p_1},
\end{cases}
\eeqq
for all $n\geq 2$.
\end{theorem}

\begin{remark}
\rm
The most interesting case is given by $p_1=1$. It follows that 
we have
\beqq
x_n(id: ~S^t_{1,1}B((0,1)^d)\to L_1((0,1)^d))
\asymp 
\left\{\begin{array}{lll}
 n^{- t} &  \text{if}\quad   t< \frac{1}{2},\\
n^{-t}(\log n)^{(d-1)(t-\frac{1}{2})} & \text{if}\quad  t> \frac{1}{2},\
\end{array}\right.
\eeqq
for all $n \ge 2$.
We are not aware of any other result concerning $s$-numbers 
(Kolmogorov numbers, approximation numbers, ...) where the exact order of 
$s_{n} (id: S^{t}_{1,1}B((0,1)^d) \to L_1((0,1)^d)) $ if $n$ tends to infinity, has been found.
A few more results concerning approximation and Kolmogorov numbers are known in case 
of the embeddings $id: S^{t}_{p_1,p_1}B((0,1)^d) \to L_1((0,1)^d) $, $p_1 >1$, 
and $id: S^{t}_{1,1}B((0,1)^d) \to L_{p_2}((0,1)^d) $, $1 < p_2 < \infty$.
E.g., in \cite{Rom9} Romanyuk has proved for $2 \le p_1 < \infty$ and $t>0$
\[
a_n(id: ~S^t_{p_1,p_1}B((0,1)^d)\to L_1((0,1)^d))
\asymp 
n^{-t}(\log n)^{(d-1)(t  +\frac{1}{2} - \frac{1}{p_1})} 
\]
for all $n \ge 2$.
\end{remark}


\subsection{A version of H\"older-Zygmund spaces (related to tensor products) as target spaces}
\label{extrem2}


As a supplement we investigate the Weyl numbers of the embeddings 
$id: \, S^{t}_{p_1,p_1}B ((0,1)^d) \to {\mathcal Z}_{\mix}^s ((0,1)^d)$,
where the spaces ${\mathcal Z}_{\mix}^s ((0,1)^d)$ are versions of 
H\"older-Zygmund spaces. Let $j \in \{1, \, \ldots \, d\}$.
For $m \in \N$, $h_j \in \re$ and $x \in \R$ we put
\[
\Delta_{h_j,j}^m f(x):= \sum_{\ell =0}^m (-1)^{m-\ell} \, \binom{m}{\ell} \, f(x_1, \ldots \, , x_{j-1}, x_j + \ell h_j, x_{j+1}, \ldots \, , x_d)\, .
\]
This is the $m$-th order difference in direction $j$. Mixed differences are defined as follows.
Let $e$ be a non-trivial subset of $\{1, \ldots \, d\}$.
For $ h  \in \R$ we define
\[
\Delta_{ h ,e}^m := \prod_{j \in e} \Delta_{h_j,j}^m \, . 
\]
Of course, here $\Delta_{h_j,j}^m \, \cdot \, \Delta_{h_\ell,\ell}^m $
has to be interpreted as $\Delta_{h_j,j}^m \, \circ \, \Delta_{h_\ell,\ell}^m $.

\begin{definition}\label{mix}
Let $s >0$. Let $m \in \N$ s.t. $m-1 \le s < m$. Then $f \in \cz^s_\mix ((0,1)^d)$ if 
\[
\| \, f\, |\cz^s_\mix ((0,1)^d)\|:= \| \, f\, |C((0,1)^d)\| + \max_{e \subset \{1, \ldots  ,d\}} \, \sup_{\|h\|_\infty \le 1} \prod_{j \in e} |h_j|^{-s}\, 
\sup_{x \in \Omega_{m,e,h}} \Big|\Delta_{ h ,e}^m f(x)\Big| < \infty\, , 
\]
where
\[
\Omega_{m,e,h} := \{x \in (0,1)^d: \quad (x_1+ \varepsilon _1 \ell_1 h_1, \, \ldots \, , x_d + 
\varepsilon_d \ell_d h_d)\in (0,1)^d \quad \forall \bar{\ell} \in \N_0^d\, , \: \| \, \bar{\ell}\, \|_\infty \le m \}\, ,
\]
and
\[
\varepsilon_j := \left\{ \begin{array}{lll}
  1 &\qquad & \mbox{if} \quad j \in e\, , 
  \\
  0 && \mbox{if} \quad j \not\in e\, . 
\end{array} \right.
\]
\end{definition}

A few properties of these spaces are obvious:
 \begin{itemize}
 \item Let $d=1$ and $m=1$, i.e., $0< s < 1$. Then $\cz^s_\mix (0,1)$ is the classical space of H\"older-continuous functions of order $s$.
 \item Let $d=1$, $s=1$ and $m=2$. Then $\cz^1_\mix (0,1)$ is the classical Zygmund space.
 \item If $f(x) = f_1 (x_1) \cdot \, \ldots \cdot f_d (x_d)$ with $f_j \in \cz_\mix (0,1)$, $j=1, \ldots, d$, then
 $f \in \cz^s_\mix ((0,1)^d)$ follows and 
 \[
 \| \, f\, |\cz^s_\mix ((0,1)^d)\| \asymp \prod_{j=1}^d \| \, f_j\, |\cz^s_\mix (0,1)\|\, .
 \]
 \item Let $f \in \cz^s_\mix ((0,1)^d)$ and define $g(x'):= f(x',0)$, where $x=(x',x_d)$, $x' \in \re^{d-1}$. Then
 $g \in \cz^s_\mix ((0,1)^{d-1})$ follows.
 \item We define $ \cz^s_\mix (\R)$ by replacing $(0,1)^d$ by $\R$ in the Def. \ref{mix}. 
 Let $E: \cz^s(0,1) \to \cz^s(\mathbb{R})$ be a linear and bounded extension operator such that 
 $E$ maps $C(0,1)$ into itself. Then
 $E \otimes \ldots \otimes E$ ($d$-fold tensor product) is well-defined on $C((0,1)^d)$ and maps this space into itself.
 Observe that $\Delta_{ h ,e}^m f(x) $ can be written as the $|e|$-fold iteration of a directional difference. As a consequence
 we obtain that $\ce_d:= E \otimes \ldots \otimes E$ maps $\cz^s_\mix ((0,1)^d)$ into itself.
 
 \end{itemize}
 
 Less obvious is the following lemma, see \cite[Rem.~2.3.4/3]{ST} or \cite{U1}.

 \begin{lemma}\label{zyg}
 Let $s>0$. Then
 \[
 \cz^s_\mix((0,1)^d) = S^s_{\infty,\infty}B((0,1)^d)
 \]
 holds in the sense of equivalent norms.
 \end{lemma}

Essentially by the same methods as used for the proof of Thm. \ref{main} one obtains the following. 

\begin{theorem}\label{satz04}
Let $s>0$ and $t> s+ \frac 1{p_1} $. Then it holds
\begin{eqnarray*}
x_{n} (id: S^{t}_{p_1,p_1}B((0,1)^d) & \to & \cz^s_\mix ((0,1)^d)) 
\\
 & \asymp & 
\left\{
\begin{array}{lll}
\frac{(\log n)^{(d-1)(t-s -\frac{1}{p_1} )}}{n^{t-s-\frac{1}{2}}} & \quad & \mbox{if}\quad 0 < p_1 \le 2 \, , 
\\
&& \\
\frac{(\log n)^{(d-1)(t -s - \frac{1}{p_1} )}}{n^{t-s -\frac{1}{p_1}}} 
& \quad & \mbox{if}\quad 2 \le p_1 \le \infty \, , 
\nonumber
\end{array}
\right.
\end{eqnarray*}
for all $n\geq 2$.
\end{theorem}

\begin{remark}\label{nochmal}
 \rm
(i) Recall that $S^{t}_{p_1,p_1}B((0,1)^d)$ is compactly embedded into $\cz^s_\mix ((0,1)^d)$ if and only if 
$t > s+1/p_1$, see \cite{Vybiral}.
\\
(ii)
Observe, that Thm. \ref{satz04} is not the limit of Thm. \ref{satz01} for $s \downarrow 0$. 
There, in Thm. \ref{satz01}, is an additional factor $(\log n)^{(d-1)/2}$
as many times in this field.
\\
(iii) If we replace $\cz^s_\mix ((0,1)^d)) $ by $S^s_{\infty,\infty}B((0,1)^d)$ in Theorem \ref{satz04}, then the restriction 
$s>0 $ becomes superfluous.
\end{remark}

Finally, we wish to mention that these methods also apply in case of approximation numbers. As a result we get the following.

\begin{theorem}\label{ap-cor}
Let $n\in \mathbb{N}, n\geq 2$ and $s>0$. Then we have
\beqq
a_{n} (id: S^{t}_{p_1,p_1}B((0,1)^d) &\to& \cz^{s}_{\mix}((0,1)^d))
\\
&\asymp&
\left\{
\begin{array}{lll}
\frac{(\log n)^{(d-1)(t-s - \frac{1}{p_1} )}}{n^{t-s -\frac{1}{p_1}}} 
& \mbox{if}\ \ 2 \le p_1 \le \infty\, , \quad t-s> \frac{1}{p_1} \, ,
\\
&
\\
\frac{(\log n)^{(d-1)(t-s- \frac{1}{p_1} )}}{n^{t-s-\frac{1}{2}}} & \mbox{if}\ \ 1 \le p_1 < 2\, , \quad 
t-s>1\, , 
\\
&
\\
\frac{(\log n)^{(d-1)(t-s- \frac{1}{p_1} )}}{n^{\frac{p_1'}{2}(t-s-\frac{1}{p_1})}} & \mbox{if}\ \ 1 < p_1 < 2\, , \quad 
1>t-s>\frac{1}{p_1}\,  ,
\end{array}
\right.
\eeqq
for all $n \ge 2$. Here $p_1'$ is the conjugate of $p_1$.
\end{theorem}

\begin{remark}
\rm
As in Remark \ref{nochmal},  if we replace $\cz^s_\mix ((0,1)^d)) $ by $S^s_{\infty,\infty}B((0,1)^d)$ in Theorem \ref{ap-cor}, 
then the restriction 
$s>0 $ becomes superfluous.
\end{remark}


\subsection{A comparison with entropy numbers}
\label{entropy}


There are good reasons to compare Weyl numbers with entropy numbers. 
Both, entropy and Weyl numbers, are tools to control the 
behaviour of eigenvalues of linear operators. 

Let us recall the definition of entropy numbers.

\begin{definition}
Let $T: \, X \to Y$ be a bounded linear operator between complex
quasi-Banach spaces, and let $n \in \N$. Then the $n$-th (dyadic) entropy
number of T is defined as
$$
e_n (T: \, X \to Y):=\inf\{ \varepsilon
>0: T(B_X) \text{ can be covered by } 2^{n-1}
\text{ balls in } Y \text{ of radius } \varepsilon\}\, ,
$$
where $B_X:= \{x \in X: \: \|x\|_X \le 1\}$ denotes the closed
unit ball of $X$.
\end{definition}

In particular, $T: \, X \to Y$ is compact if and only if
$\lim_{n \to \infty} e_n (T) = 0 \,$. For details and basic
properties like multiplicativity, additivity, behaviour under
interpolation etc. we refer to the monographs
 \cite{CS,ET,Koe,Pi1}.
Most important for us is the Carl-Triebel inequality
which states
\beqq
|\lambda_n (T)| \le \sqrt{2}\, e_n (T)\, ,
\eeqq
cf. Carl, Triebel \cite{CT} (see also the monographs \cite{CS} and \cite{ET}).

Entropy numbers of embeddings $id:\ S^t_{p_1,p_1}B((0,1)^d) \hookrightarrow L_{p_2}((0,1)^d)$
have been investigated in Vybiral \cite{Vybiral}. The picture is less complete than in case of Weyl numbers.
Only for sufficiently large smoothness the behaviour is exactly known.
For $0 < p_1 \le \infty$ and $1 < p_2 < \infty$ we have
\[
e_n(id) \asymp n^{-t}(\log n)^{(d-1)(t+\frac{1}{2}-\frac{1}{p_1})}\qquad 
\text{if } \quad t>\max \Big(0, \frac{1}{p_1}-\frac{1}{2}, \frac{1}{p_1}-\frac{1}{p_2}\Big)\, ,  \qquad n \ge 2\, .
\]

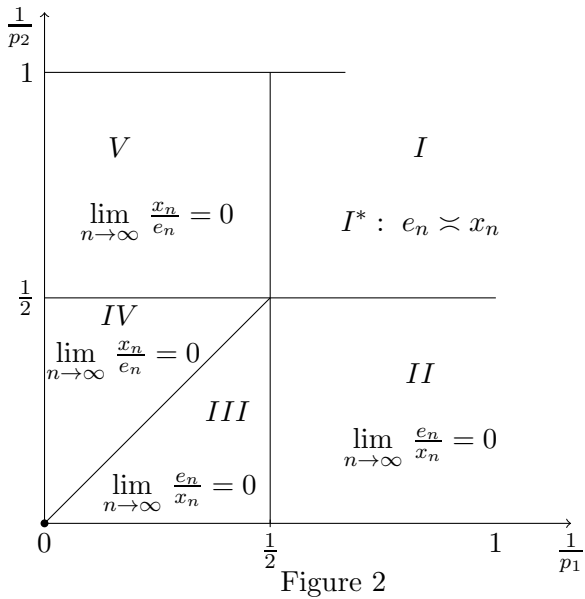
\begin{minipage}[t]{9cm}
$$
\begin{tikzpicture}
\fill (0,0) circle (1.5pt);
\draw[->, ](0,0) -- (7,0);
\draw[->, ] (0,0) -- (0,6.8);
\draw (3,0)-- (3,6);
\draw (0,0) -- (3,3);
\draw (0,3) -- (6,3);
\draw (0,6) -- (4,6);
\node[below] at (0,0) {$0$};
\node [below] at (3,0) {$\frac{1}{2}$};
\node [left] at (0,3) {$\frac{1}{2}$};
\node [left] at (0,6) {$ 1$};
\node [left] at (0,6.6) {$\frac{1}{p_2}$};
\node [below] at (6,0) {$1$};
\draw (3, -0.05) -- (3, 0.05);
\node [below] at (7,0) {$\frac{1}{p_1}$};
\node [] at (5,4) {$I^*: ~ e_n \asymp x_n$};
\node [] at (5,5) {$I$};
\node [] at (5,1) {$\lim\limits_{n \to \infty}\, \frac{e_n}{x_n} = 0 $};
\node [] at (5,2) {$II$};
\node [] at (1,5) {$V$};
\node [] at (1.5,4) {$\lim\limits_{n \to \infty}\, \frac{x_n}{e_n} = 0 $};
\node [right] at (2,3/2) {$III$};
\node [] at (1.8,0.4) {$\lim\limits_{n \to \infty}\, \frac{e_n}{x_n} = 0 $};
\node [above] at (1,2.5) {$IV$};
\node [] at (1.05,2.15) {$\lim\limits_{n \to \infty}\, \frac{x_n}{e_n} = 0 $};
\node [right] at (3,-0.8) {Figure 2};
\end{tikzpicture}
$$
\end{minipage}
\hfill
\begin{minipage}[t]{4cm}
{~}\\ 
We use Figure 2 to explain the different behaviour of entropy and Weyl numbers.
Weyl numbers are essentially smaller than entropy numbers in regions IV and V, entropy numbers are 
essentially smaller than Weyl numbers in regions II and III, and they show a similar behaviour in region I$^*$.
\end{minipage}

\begin{remark}
 \rm
(i) Further estimates of the decay of entropy numbers related to embeddings 
$id:\ S^t_{p_1,q_1}B((0,1)^d) \hookrightarrow L_{p_2}((0,1)^d)$ 
($id:\ S^t_{p_1} W((0,1)^d) \hookrightarrow L_{p_2}((0,1)^d)$)
can be found in Belinsky \cite{Be}, Dinh D\~ung \cite{DD}, and Temlyakov \cite{Tem2}.
\\
(ii) There are many contributions dealing with  the behaviour of 
$d_n(id:\ S^t_{p_1,q_1}B((0,1)^d) \hookrightarrow L_{p_2}((0,1)^d))$ (Kolmogorov numbers)
and 
$a_n(id:\ S^t_{p_1,q_1}B((0,1)^d) \hookrightarrow L_{p_2}((0,1)^d))$. 
However, the picture is much less complete than in case of Weyl numbers.
We refer to Bazarkhanov \cite{Baz7} for the most recent publication in this direction.
The topic itself has been investigated at various places over the last 30 years, see, e.g.,  
Temylakov \cite{Tem2,Tem}, Galeev \cite{Ga2,Ga,Ga1} and Romanyuk \cite{Rom6,Rom7,Rom4,Rom8,Rom1,Rom2,Rom5,Rom3,Rom9}.
\end{remark}


\section{Weyl numbers - basic properties}\label{basic}


Weyl numbers are special $s$-numbers. 
For later use we recall this general notion following Pietsch \cite[2.2.1]{Pi3}
(note that this differs slightly from earlier definitions in the literature).\\

Let $X,Y,X_0,Y_0$ be quasi-Banach spaces.
As usual, $\mathcal L(X,Y)$ denotes the space of all continuous linear operators from $X$ to $Y$. Finally, let $Y$ be $p$-Banach space for some $p \in (0,1]$, i.e.,
\be\label{dreieck}
\|\, x+y\, |Y\|^{p} \le \|\, x\, |Y\|^{p} + \|\, y \, |Y\|^{p} \qquad \mbox{for all}\quad x,\, y \in Y\, .
\ee
An $s$-function is a map $s$ assigning to every operator $T\in \mathcal L(X,Y)$ a scalar sequence $(s_n(T))$ such that the following conditions are satisfied:
\begin{description}
\item[(a)] $\|T\|=s_1(T)\geq s_2(T)\geq...\geq 0 $ for all $T\in \mathcal L(X,Y)$;
\item[(b)] $ s_{n+m-1}^p(S+T)\leq s_n^p(S)+ s_m^p(T) $ for $S,T\in \mathcal L(X,Y)$ and $m,n=1,2, \ldots \, $;
\item[(c)] $s_n(BTA)\leq \|B\| \, \cdot \, s_n(T) \, \cdot \, \|A\|$ for $A\in \mathcal L(X_0,X)$, $T\in \mathcal L(X,Y)$, $B\in \mathcal L(Y,Y_0)$;
\item[(d)] $s_n(T)=0$ if $\text{rank}(T)<n$ for all $n\in \N$; 
\item[(e)] $s_n(id: \ell_2^n\to \ell_2^n)=1$ for all $n\in \N$.
\end{description}
We will refer to $(a)$ as monotonicity, to $(b)$ as additivity, to $(c)$ as ideal property, to $(d)$ as the rank property
and to $(e)$ as normalization (norm-determining property) of the $s$-numbers.
\\
Sometimes a further property is of some use.
Let $Z$ be a quasi-Banach space.
An $s$-function is called multiplicative if
\\
${\bf (f)}$ $ s_{n+m-1}(ST)\leq s_n(S)\, s_m(T) $ for $T\in \mathcal L(X,Y)$, $S\in \mathcal L(Y,Z)$ and $m,n=1,2, \ldots \, $.

\subsection*{Examples}

The following numbers are $s$-numbers:
\begin{enumerate}
\item Kolmogorov numbers are multiplicative $s$-numbers, see, e.g., \cite[Thm. 11.9.2]{Pi1}.
\item 
Approximation numbers are multiplicative $s$-numbers, see, e.g., \cite[2.3.3]{Pi3}.
\item The $n$-th Gelfand number of the linear operator $T \in \mathcal L(X,Y)$ is defined to be 
$$ c_n (T) := \inf\Big\{\|\, T\, J_M^X\, \|: \quad {\rm codim\,}(M)< n\Big\},$$
where $J_M^X:M\to X$ refers to the canonical injection of $M$ into $X$. 
Gelfand numbers are multiplicative $s$-numbers, see, e.g., \cite[Prop.~2.4.8]{Pi3}.
\item Weyl numbers are multiplicative $s$-numbers, see \cite[2.4.14,~2.4.17]{Pi3}.
\end{enumerate}

Entropy numbers do not belong to the class of $s$-numbers since they do not satisfy $(d)$.

\begin{remark} \label{weylextra}
\rm
There is an alternative way to calculate the $n$-th Weyl number. Indeed, for $T \in \mathcal L(X,Y)$ it holds
$$ x_n(T):=\sup \Big\{c_n(TA):\quad A\in \mathcal L(\ell_2,X),\ \|A\|\leq 1\Big\}\, , $$
see Pietsch \cite{Pi2}.
\end{remark}

\subsection*{Interpolation properties of Weyl numbers}

For later use we add the following assertion concerning interpolation properties of Weyl numbers.

\begin{theorem}\label{inter}
Let $0<\theta<1$.
Let $X,Y, Y_0,Y_1$ be a quasi-Banach spaces. 
Further we assume  $Y_0\cap Y_1\hookrightarrow  Y$ and the existence of  a positive constant $C$ such that
\be\label{weylextra3}
\| y|Y\| \leq C \, \| y|Y_0\|^{1-\theta}\|y|Y_1\|^{\theta}\qquad  \text{for all}\quad  y\in Y_0\cap Y_1.
\ee
Then, if 
\[
T\in  \mathcal{L}(X,Y_0) \cap  \mathcal{L}(X,Y_1) \cap  \cl (X,Y)
\]
we obtain 
\beqq
x_{n+m-1}(T:~X\to Y)\le C\,  x_n^{1-\theta}(T:~ X\to Y_0)\, x_m^{\theta}(T:~X\to Y_1)
\eeqq
for all $n,m \in \N$. Here $C$ is the same constant as in \eqref{weylextra3}.
\end{theorem}

\begin{remark}
 \rm
Interpolation properties of Kolmogorov and Gelfand numbers have been studied by Triebel \cite{Tr70}.
Theorem \ref{inter} and Theorem \ref{inte2} below show that Gelfand and Weyl numbers share the same interpolation properties. 
\end{remark}


\section{Tensor product Besov spaces and spaces of dominating mixed smoothness}
\label{domino}


As mentioned before tensor product Besov spaces can be interpreted as special cases of the scale of  Besov spaces of 
dominating mixed smoothness.
For us it will be convenient to introduce these classes  of dominating mixed smoothness by means of wavelets.
In the Appendix B below we recall the probably better known Fourier-analytic definition. 
In addition we shall introduce Lizorkin-Triebel spaces of dominating mixed smoothness. They will be used  in our proofs of the main results for Besov spaces. 

Let $\bar{\nu} =(\nu_1, \ldots ,\nu_d) \in \mathbb{N}_0^d$ and $\bar{m}= (m_1, \ldots , m_d)\in \mathbb{Z}^d$. 
Then we put $2^{-\bar{\nu}}\bar{m}=(2^{-\nu_1}m_1,...,2^{-\nu_d}m_d)$ and
\[
Q_{\bar{\nu},\bar{m}} := \Big\{x\in \R: \quad 2^{-\nu_\ell} \, m_\ell < x_\ell < 2^{-\nu_\ell}\, (m_\ell+1)\, , \: \ell = 1, \, \ldots \, , d\Big\} \, .
\]
By $\chi_{{\bar{\nu},\bar{m}}}(x)$ we denote the characteristic function of $Q_{\bar{\nu},\bar{m}}$. 
First we have to introduce some sequence spaces.

\begin{definition}\label{sequence1}
If $0<p,q\leq \infty$, $t\in \mathbb{R}$ and
$\lambda:=\lbrace \lambda_{\bar{\nu},\bar{m}}\in\mathbb{C}:\bar{\nu}\in \mathbb{N}_0^d,\ \bar{m}\in \mathbb{Z}^d \rbrace$,
then we define
$$s_{p,q}^t b := \Big\lbrace\lambda: \| \lambda|s_{p,q}^t b\| =
\Big(\sum_{\bar{\nu}\in \mathbb{N}_0^d}2^{|\bar{\nu}|_1 (t-\frac{1}{p})q}\big(\sum_{\bar{m}\in \mathbb{Z}^d}|\lambda_{\bar{\nu},\bar{m}} 
|^p\big)^{\frac{q}{p}}\Big)^{\frac{1}{q}}<\infty \Big\rbrace$$
and, if $p< \infty$, 
$$s_{p,q}^tf=\Big\lbrace\lambda: \| \lambda|s_{p,q}^tf\| =
\Big\|\Big(\sum_{\bar{\nu}\in \mathbb{N}_0^d}\sum_{\bar{m}\in \mathbb{Z}^d}|2^{|\bar{\nu}|_1 t}
\lambda_{\bar{\nu},\bar{m}}\chi_{\bar{\nu},\bar{m}}(.) 
|^q\Big)^{\frac{1}{q}}\Big|L_p(\mathbb{R}^d)\Big\|<\infty \Big\rbrace$$
with the usual modification for $p$ or/and q equal to $\infty$.
\end{definition}

\begin{remark}\label{lift}
\rm Let $\sigma \in \re$.
For later use we mention that the mapping 
\be\label{lift1}
J_\sigma:~ (\lambda_{\bar{\nu},\bar{m}})_{\bar{\nu},\bar{m}} \mapsto (2^{\sigma |\bar{\nu}|_1}\, \lambda_{\bar{\nu},\bar{m}})_{\bar{\nu},\bar{m}} 
\ee
yields an isomorphism of $s^t_{p,q}a$ onto $s^{t-\sigma}_{p,q}a$, $a \in \{b,f\}$.
\end{remark}

Now we recall wavelet bases of Besov and Lizorkin-Triebel spaces of dominating mixed smoothness.
Let $N \in \N$. Then there exists $\psi_0, \psi_1 \in C^N(\re) $, compactly supported, 
\[
\int_{-\infty}^\infty t^m \, \psi_1 (t)\, dt =0\, , \qquad m=0,1,\ldots \, , N\, , 
\]
such that
$\{ 2^{j/2}\, \psi_{j,m}: \quad j \in \N_0, \: m \in \zz\}$, where
\[
\psi_{j,m} (t):= \left\{ \begin{array}{lll}
\psi_0 (t-m) & \qquad & \mbox{if}\quad j=0, \: m \in \zz\, , 
\\
\sqrt{1/2}\, \psi_1 (2^{j-1}t-m) & \qquad & \mbox{if}\quad j\in \N\, , \: m \in \zz\, , 
            \end{array} \right.
\]
is an orthonormal basis in $L_2 (\re)$, see \cite{woj}. 
We put
\[
\Psi_{\bar{\nu}, \bar{m}} (x) := \prod_{\ell=1}^d \psi_{\nu_\ell, m_\ell} (x_\ell)\, . 
\]
Then 
\[
\Psi_{\bar{\nu}, \bar{m}}\, , \qquad \bar{\nu} \in \N_0^d, \, \bar{m} \in \Z\, ,
\]
is a tensor product wavelet basis of $L_2 (\R)$. Vybiral \cite{Vybiral} has proved the following.

\begin{lemma}\label{wavelet}
Let $0< p,q \le\infty$ and $t\in \re$. 
\\
{\rm (i)}
There exists $N=N(t,p) \in \N$ s.t. the mapping 
\beqq
{\mathcal W}: \quad f \mapsto (2^{|\bar{\nu}|_1}\langle f, \Psi_{\bar{\nu},\bar{m}} \rangle)_{\bar{\nu} \in \N_0^d\, , \, \bar{m} \in \Z} 
\eeqq
is an isomorphism of $S^t_{p,q}B(\R)$ onto $s^t_{p,q}b$.
\\
{\rm (ii)} Let $p <\infty$. Then
there exists $N=N(t,p,q) \in \N$ s.t. the mapping ${\mathcal W}$
is an isomorphism of $S^t_{p,q}F(\R)$ onto $s^t_{p,q}f$.
\end{lemma}

\subsection*{Spaces on $\Omega$}

We put $\Omega := (0,1)^d$.
For us it will be convenient to define spaces on $\Omega $ by restrictions.
We shall need the set $D'(\Omega)$, consisting of all complex-valued distributions on $\Omega$.

\begin{definition} \label{defomega}
{\rm (i)} Let $ 0< p,q\leq \infty$ and $t\in \re$. Then
   $S^{t}_{p,q}B((0,1)^d)$ is the space of all $f\in D'(\Omega)$ such that there exists a distribution $g\in
   S^{t}_{p,q}B(\R)$ satisfying $f = g|_{\Omega}$. It is endowed with the quotient norm
   $$
      \|\, f \, |S^{t}_{p,q}B((0,1)^d)\| = \inf \Big\{ \|g|S^{t}_{p,q}B(\R)\|~: \quad g|_{\Omega} =
      f \Big\}\,.
   $$
{\rm (ii)} Let $ 0< p < \infty$, $0< q\leq \infty$ and $t\in \re$. Then
   $S^{t}_{p,q}F((0,1)^d)$ is the space of all $f\in D'(\Omega)$ such that there exists a distribution $g\in
   S^{t}_{p,q}F(\R)$ satisfying $f = g|_{\Omega}$. It is endowed with the quotient norm
   $$
      \|\, f \, |S^{t}_{p,q}F((0,1)^d)\| = \inf \Big\{ \|g|S^{t}_{p,q}F(\R)\|~: \quad g|_{\Omega} =
      f \Big\}\,.
   $$   
\end{definition}

Several times we shall work with the following consequence of this definition in combination with Lemma \ref{wavelet}.
Let $t,p$ and $q$ be fixed.
Let the wavelet basis $\Psi_{\bar{\nu}, \bar{m}}$ be admissible in the sense of Lemma \ref{wavelet}.
We put
\be\label{ws-40}
A_{\bar{\nu}}^{\Omega}:= 
\Big\{\bar{m}\in \mathbb{Z}^d: \quad \supp \Psi_{\bar{\nu},\bar{m}} \cap\Omega \neq \emptyset\Big\}\, ,\qquad \bar{\nu}\in\mathbb{N}_0^d\, .
\ee
For given $ f \, \in S_{p,q}^t A(\Omega)$, $A\in \{B,F\}$, let $\ce f$ be an element of $ S_{p,q}^t A(\R)$ s.t. 
\[
\| \, \ce f \, | S_{p,q}^t A(\R)\| \le 2 \, \| \, f \, | S_{p,q}^t A(\Omega)\|
\qquad \mbox{and} \qquad (\ce f)_{|_\Omega} = f \, .
\]
We define
\[
g:= \sum_{\bar{\nu} \in \N_0^d} \sum_{\bar{m} \in A_{\bar{\nu}}^{\Omega}} 2^{|\bar{\nu}|_1} \, \langle \ce f, \Psi_{\bar{\nu},\bar{m}} \rangle \, \Psi_{\bar{\nu}, \bar{m}}\, .
\]
Then it follows that $g \in S_{p,q}^t A(\R)$, $g_{|_\Omega} = f $, 
\[
\supp g \subset \{x \in \R: ~ \max_{j=1, \ldots \, , d} |x_j| \le c_1\} \quad \mbox{and} \qquad 
\| \, g \, | S_{p,q}^t A(\R)\| \le c_2 \, \| \, f \, | S_{p,q}^t A(\Omega)\|\, .
\]
Here $c_1,c_2$ are independent of $f$.


\subsection*{Tensor products of Besov spaces}


Tensor products of Besov spaces have been investigated in \cite{Ha}, \cite{SUt} and \cite{SUspline}.
We recall some results from \cite{SUt} and \cite{SUspline}.
For the basic notions of tensor products used here we refer to \cite{LiCh} and \cite{DF}. By $\alpha_p$ we denote the $p$-nuclear norm
and by $\gamma_p$ the projective tensor $p$-norm.

\begin{theorem}[Tensor products of Besov spaces on the interval]\label{thm2l}\quad\\
Let $d\ge 1$ and let $t \in \re$.\\ 
{\rm (i)}
Let $1<p<\infty$. Then the following formula
\begin{eqnarray*} 
B^{t}_{p,p}(0,1)\otimes_{\alpha_p} S^{t}_{p,p}B((0,1)^d)
 & = & S^{t}_{p,p}B((0,1)^d) \otimes_{\alpha_p} B^{t}_{p,p}(0,1) 
\\
& = & S^{t}_{p,p}B((0,1)^{d+1})
\end{eqnarray*}
holds true in the sense of equivalent norms.\\
{\rm (ii)} Let $0 <p\le 1$. Then the following formula
\begin{eqnarray*} B^{t}_{p,p}(0,1)\otimes_{\gamma_p} S^{t}_{p,p}B((0,1)^d) 
& = & S^{t}_{p,p}B((0,1)^d) \otimes_{\gamma_p} B^{t}_{p,p}(0,1)
\\
& = & S^{t}_{p,p}B((0,1)^{d+1})
\end{eqnarray*} holds true in the sense of equivalent quasi-norms.
\end{theorem}

\begin{remark}
 \rm
For easier notation we put $\gamma_p := \alpha_p$ if $ 1 < p< \infty$.
One can iterate the process of taking tensor products.
Defining for $m>2$
\[
X_1 \otimes_{\gamma_p} X_2 \otimes_{\gamma_p} \, \ldots \,
\otimes_{\gamma_p} X_m := X_1 \otimes_{\gamma_p} \Big( \ldots
X_{m-2} \otimes_{\gamma_p} \, (X_{m-1} \otimes_{\gamma_p} X_m )\Big)
\]
we obtain an interpretation of $S^{t}_{p,p} B((0,1)^d)$, $0< p <\infty$, as an
iterated tensor product of univariate Besov spaces, namely
\[
S^{t}_{p,p} B((0,1)^d) = B^{t}_{p,p}(0,1)\otimes_{\gamma_p} \, \ldots \, \otimes_{\gamma_p} B^{t}_{p,p}(0,1)\, , \qquad 0 < p< \infty\, .
\]
The iterated tensor products, considered in this paper, do not depend on the order of the tuples which are formed during
the process of calculating $X_1 \otimes_{\gamma_p} X_2 \otimes_{\gamma_p} \, \ldots \, \otimes_{\gamma_p} X_m $, i.e.,
\[
( X_1 \otimes_{\gamma_p} X_2) \otimes_{\gamma_p} \, X_3 =
 X_1 \otimes_{\gamma_p} ( X_2 \otimes_{\gamma_p} \, X_3 )\, .
\]
\end{remark}

\noindent
Consequently, if $p < \infty$, we may deal with $S^{t}_{p,p} B((0,1)^d)$ instead of 
$B^{t}_{p,p}(0,1)\otimes_{\gamma_p} \, \ldots \, \otimes_{\gamma_p} B^{t}_{p,p}(0,1)$.


\section{Weyl numbers of embeddings of sequence spaces}\label{sequence}


In this section we will estimate the behavior of Weyl numbers of the identity mapping
\[ 
id^*: \quad s_{p_0,p_0}^{t,\Omega}b\to s_{p,2}^{0,\Omega}f \, .
\]
Here we assume that $p_0$ varies in $(0,\infty]$ and $p$ in $(0,\infty)$.


\subsection{Preparations}\label{prepa}


For technical reasons we need a few more sequence spaces. Recall, $A_{\bar{\nu}}^{\Omega} $ has been defined in \eqref{ws-40}.

\begin{definition}
If $0<p\leq \infty$, $0< q\leq \infty$, $t\in \mathbb{R}$ and
\[
\lambda=\lbrace \lambda_{\bar{\nu},\bar{m}}\in\mathbb{C}:\bar{\nu}\in \mathbb{N}_0^d,\ \bar{m}\in A_{\bar{\nu}}^{\Omega} \rbrace \, ,
\]
then we define
\[
s_{p,q}^{t,\Omega}b : =\Big\lbrace\lambda: \quad \| \lambda|s_{p,q}^{t,\Omega}b\| =\Big(\sum_{\bar{\nu}\in \mathbb{N}_0^d}2^{|\bar{\nu}|_1(t-\frac{1}{p})q}\big(\sum_{\bar{m}\in A_{\bar{\nu}}^{\Omega}}
|\lambda_{\bar{\nu},\bar{m}} |^p\big)^{\frac{q}{p}}\Big)^{\frac{1}{q}}<\infty \Big\rbrace
\]
and, if $p< \infty$, 
\[
s_{p,q}^{t,\Omega}f :=\Big\lbrace\lambda: \quad \| \lambda|s_{p,q}^{t,\Omega}f\| =\Big\|\Big(\sum_{\bar{\nu}\in \mathbb{N}_0^d}\sum_{\bar{m}\in A_{\bar{\nu}}^{\Omega}}|2^{|\bar{\nu}|_1 t}\lambda_{\bar{\nu},\bar{m}}
\chi_{\bar{\nu},\bar{m}}(.) |^q\Big)^{\frac{1}{q}}\Big|L_p(\mathbb{R}^d)\Big\|<\infty \Big\rbrace\, .
\]
\end{definition}

In addition we need the following sequence of subspaces.

\begin{definition}
If $0<p\leq \infty$, $0< q\leq \infty$, $t\in \mathbb{R}$, $\mu\in \mathbb{N}_0$ and 
$$\lambda=\lbrace \lambda_{\bar{\nu},\bar{m}}\in\mathbb{C}:\bar{\nu}\in \mathbb{N}_0^d,\ |\bar{\nu}|_1=\mu,\ \bar{m}\in A_{\bar{\nu}}^{\Omega} \rbrace \, , $$
then we define
$$(s_{p,q}^{t,\Omega}b)_{\mu}=\Big\lbrace\lambda: \| \lambda|(s_{p,q}^{t,\Omega}b)_{\mu}\| =
\Big(\sum_{|\bar{\nu}|_1=\mu}2^{|\bar{\nu}|_1(t-\frac{1}{p})q}\big(\sum_{\bar{m}\in A_{\bar{\nu}}^{\Omega}}|\lambda_{\bar{\nu},\bar{m}} |^p\big)^{\frac{q}{p}}\Big)^{\frac{1}{q}}<\infty \Big\rbrace$$
and, if $p< \infty$, 
$$(s_{p,q}^{t,\Omega}f)_{\mu}=\Big\lbrace\lambda: \| \lambda|(s_{p,q}^{t,\Omega}f)_{\mu}\| =
\Big\|\Big(\sum_{|\bar{\nu}|_1=\mu}\sum_{\bar{m}\in A_{\bar{\nu}}^{\Omega}}|2^{|\bar{\nu}|_1t}\lambda_{\bar{\nu},\bar{m}}\chi_{\bar{\nu},\bar{m}}(.) 
|^q\Big)^{\frac{1}{q}}\Big|L_p(\mathbb{R}^d)\Big\|<\infty \Big\rbrace \, .$$
\end{definition}

To avoid repetitions we shall use $s^{t}_{p,q}a$, $s^{t,\Omega}_{p,q}a$, $(s^{t,\Omega}_{p,q}a)_\mu$
with $a \in \{b,f\}$ in case that an assertion holds for both scales simultaneously.
Here in this paper we do not deal with the spaces $s^{t,\Omega}_{\infty,q}f$ and $(s^{t,\Omega}_{\infty,q}f)_\mu $.
But we will use the convention that, whenever $s^{t,\Omega}_{\infty,q}a$ or $(s^{t,\Omega}_{\infty,q}a)_\mu $ occur, this has to 
be interpreted as $s^{t,\Omega}_{\infty,q}b$ and $(s^{t,\Omega}_{\infty,q}b)_\mu $. 
The two following elementary lemmas are taken from \cite[Lemma 3.10]{Vybiral} and \cite[Lemma 6.4.2]{Hansen}.

\begin{lemma}\label{ba1}
\begin{enumerate}
\item We have
\[
 |A_{\bar{\nu}}^{\Omega}| \asymp 2^{|\bar{\nu}|_1},\qquad D_{\mu}:= \sum_{|\bar{\nu}|_1=\mu} |A_{\bar{\nu}}^{\Omega}| \asymp \mu^{d-1}2^{\mu} 
\]
with equivalence constants independent of $\bar{\nu}\in \mathbb{N}_0^d$ and $\mu\in \mathbb{N}_0$. 
\item Let $0<p < \infty$ and $t\in \mathbb{R}$. Then
$$ s_{p,p}^{t,\Omega}f = s_{p,p}^{t,\Omega}b$$
and 
\[
(s_{p,p}^{t,\Omega}f)_{\mu}=(s_{p,p}^{t,\Omega}b)_{\mu} =2^{\mu(t-\frac{1}{p})}\ell_p^{D_{\mu}},\ \ \mu\in \mathbb{N}_0\, ,
\]
with the obvious interpretation for the quasi-norms.
\end{enumerate}
\end{lemma}

\begin{lemma}\label{ba2}
\begin{enumerate}
\item Let $0<p_0, p \leq \infty$ and $0<q\leq\infty$. Then 
\[
\| \, id^*_{\mu} \, : (s^{t,\Omega}_{p_0,q}a)_\mu \to (s^{t,\Omega}_{p,q}a)_\mu\|
\asymp 2^{\mu(\frac{1}{p_0}-\frac{1}{p})_+}
\]
with equivalence constants independent of $\mu \in \N_0$.
\item Let $0<q_0,q\leq \infty$ and $0<p\leq \infty$. Then
$$ \| id^*_{\mu}: (s_{p,q_0}^{t,\Omega}a)_{\mu}\to (s_{p,q}^{t,\Omega}a)_{\mu} \|\asymp \mu^{(d-1)(\frac{1}{q}-\frac{1}{q_0})_+}$$
with equivalence constants independent of $\mu \in \N_0$.
\end{enumerate}
\end{lemma}

\begin{corollary}\label{ba2-1}
Let $0<p_0,p,q_0,q\leq \infty$ and $t\in \mathbb{R}$. Then
\[
\| \, id^*_{\mu} \, : (s^{t,\Omega}_{p_0,q_0}a)_\mu \to (s^{0,\Omega}_{p,q}a)_\mu\|
\lesssim 2^{\mu\big(-t+(\frac{1}{p_0}-\frac{1}{p})_+\big)}\mu^{(d-1)(\frac{1}{q} - \frac{1}{q_0})_+},
\]
with a constant behind $\lesssim$ independent of $\mu$.
\end{corollary}

\begin{proof}
This is an immediate consequence of Lemma \ref{ba2}. 
\end{proof}

Sometimes the previous estimate can be improved.

\begin{lemma}\label{ba3} Let $0<p_0<p<\infty$, $0<q_0,q\leq \infty$ and $t \in \re$. Then
$$ \|id^*_{\mu}: (s_{p_0,q_0}^{t,\Omega}f)_{\mu}\to (s_{p,q}^{0,\Omega}f)_{\mu} \| \lesssim 2^{\mu(-t+\frac{1}{p_0}-\frac{1}{p})}. $$
\end{lemma}

\begin{proof}
This assertion is contained in \cite{Hansen}. Since this phd is not published we give a proof.
Let $\lambda $ be a sequence such that 
$\lambda_{\bar{\nu}, \bar{m}} =0 $
if $|\bar{\nu}|_1 \neq \mu$.
Since $p_0<p$ the Sobolev-type embedding yields
$$s^{t,\Omega}_{p_0,q_0}f\hookrightarrow s_{p,q}^{t-\frac{1}{p_0}+\frac{1}{p},\Omega}f \, , $$
see \cite[Thm.~2.4.1]{ST}($d=2$) or \cite{HV}, we have
\beqq
\| \lambda|(s_{p,q}^{0,\Omega}f)_{\mu} \| & = & \| \lambda|s_{p,q}^{0,\Omega}f \|
=2^{\mu(-t+\frac{1}{p_0}-\frac{1}{p})}\| \lambda|s_{p,q}^{t-\frac{1}{p_0}+\frac{1}{p},\Omega}f \|
\\
& \lesssim & 2^{\mu(-t+\frac{1}{p_0}-\frac{1}{p})}\| \lambda|s_{p_0,q_0}^{t,\Omega}f \|
= 2^{\mu(-t+\frac{1}{p_0}-\frac{1}{p})}\| \lambda|(s_{p_0,q_0}^{t,\Omega}f )_{\mu}\|\, .
\eeqq
This proves the claim.
\end{proof}


\subsection{Weyl numbers of embeddings of sequence spaces related to spaces of dominating mixed smoothness - preparations}
\label{sequence2}


For $\mu \in \N_0$ we define 
$$ id_{\mu}^* :\ s_{p_0,p_0}^{t,\Omega}b\to s_{p,2}^{0,\Omega}f\, , $$
where 
$$
(id_{\mu}^*\lambda)_{\bar{\nu},\bar{m}} := \begin{cases}
\lambda_{\bar{\nu},\bar{m}}&\text{if}\ |\bar{\nu}|_1=\mu,\\
0&\text{otherwise}.
\end{cases}
$$
The main idea of our proof is the following splitting 
of $id^* :\ s_{p_0,p_0}^{t,\Omega}b\to s_{p,2}^{0,\Omega}f$
 into a sum of identities between building blocks
\be\label{ws-15}
id^* = \sum_{\mu=0}^{\infty}\, id_{\mu}^* = \sum_{\mu=0}^{J}\, id_{\mu}^* + \sum_{\mu=J+1}^{L}\, id_{\mu}^* + \sum_{\mu=L+1}^{\infty}\, id_{\mu}^*,
\ee
where $ J$ and $L$ are at our disposal. These numbers $J$ and $L$ will be chosen in dependence on the parameters.
Let us mention that a similar splitting has been used by Vybiral \cite{Vybiral} for the estimates of related entropy numbers.\\
The additivity and the monotonicity of the Weyl numbers and the quasi-triangle inequality \eqref{dreieck} yield
\be\label{ws-16}
x^{\rho}_n(id^*)\leq \sum_{\mu=0}^{J}x^{\rho}_{n_{\mu}}(id_{\mu}^*)+\sum_{\mu=J+1}^{L}x^{\rho}_{n_{\mu}}(id_{\mu}^*)+\sum_{\mu=L+1}^{\infty} \|id_{\mu}^*\|^{\rho},
\qquad \rho :=\min (1,p)\, , 
\ee
where $n-1 = \sum_{\mu=0}^{L}(n_{\mu}-1)$. Of course, $\|id_{\mu}^*\| = \| \, id_{\mu}^* \, : (s_{p_0,p_0}^{t,\Omega}f)_{\mu}\to (s_{p,2}^{0,\Omega}f)_{\mu}\|$. 
For brevity we put 
$$\alpha=t-\bigg(\frac{1}{p_0}-\frac{1}{p}\bigg)_+ \, .$$
Then by Corollary \ref{ba2-1}, we have
$$ \|id_{\mu}^* \|\lesssim 2^{-\mu\alpha}\, \mu^{(d-1)(\frac{1}{2}-\frac{1}{p_0})_+}, $$
which results in the estimate
\be\label{ws-17}
\sum_{\mu=L+1}^{\infty} \|id_{\mu}^*\|^{\rho}\lesssim 2^{-L\alpha{\rho}}L^{(d-1){\rho}(\frac{1}{2}-\frac{1}{p_0})_+} \, .
\ee
Now we choose $n_{\mu}$ 
\be\label{ws-18}
n_{\mu} := D_{\mu}+1,\qquad \mu=0,1,....,J \, .
\ee
Then we get
\be\label{ws-23}
\sum_{\mu=0}^{J}n_{\mu} \, \asymp\, \sum_{\mu=0}^{J}\mu^{(d-1)}2^{\mu} \, \asymp\, J^{d-1}2^J
\ee
and $x_{n_{\mu}}(id_{\mu}^*)=0 $, see the rank property of the $s$-numbers, 
which implies 
\be\label{ws-19}
\sum_{\mu=0}^{J}x^{\rho}_{n_{\mu}}(id_{\mu}^*)=0 \, .
\ee
Summarizing \eqref{ws-16}-\eqref{ws-19} we have found
\begin{equation}\label{sum1}
x^{\rho}_n(id^*)\lesssim \sum_{\mu=J+1}^{L} \, x^{\rho}_{n_{\mu}}(id_{\mu}^*)+ 2^{-L\alpha{\rho}}L^{(d-1){\rho}(\frac{1}{2}-\frac{1}{p_0})_+}\, .
\end{equation}
Now we turn to the problem to reduce the estimates for the Weyl numbers $x_{n_{\mu}}(id_{\mu}^*)$ to estimates for 
$x_{n}(id_{p_0,p}^m)$. 

\begin{proposition}\label{wichtig1}
Let $0<p_0\leq \infty$ and $t\in \mathbb{R}$. Then we have the following assertions.
\begin{enumerate}
\item If $0< p\leq 2$, then 
\begin{equation}\label{case1}
\mu^{(d-1)(-\frac{1}{p}+\frac{1}{2})}2^{\mu(-t+\frac{1}{p_0}-\frac{1}{p})} \, x_n(id_{p_0,p}^{D_{\mu}} )
\lesssim x_n(id_{\mu}^*) \lesssim 2^{\mu(-t+\frac{1}{p_0}-\frac{1}{2})}\, x_n(id_{p_0,2}^{D_\mu}).
\end{equation}
\item If $2\leq p<\infty$, then 
\begin{equation}\label{case2}
2^{\mu(-t+\frac{1}{p_0}-\frac{1}{2})}\, x_n(id_{p_0,2}^{D_{\mu}} )\lesssim x_n(id_{\mu}^*)\lesssim \mu^{(d-1)(\frac{1}{2}-\frac{1}{p})}2^{\mu(-t+\frac{1}{p_0}-\frac{1}{p})}
\, x_n(id_{p_0,p}^{D_\mu})\, .
\end{equation}
\end{enumerate}
\end{proposition}

\begin{proof}
{\em Step 1.} Estimate from above.
We define $\delta:=\max(p,2)$ and consider the following diagram:

\tikzset{node distance=4cm, auto}

\begin{center}
\begin{tikzpicture}
 \node (H) {$(s^{t,\Omega}_{p_0,p_0}b)_\mu$};
 \node (L) [right of =H] {$(s^{0,\Omega}_{p,2}f)_\mu $};
 \node (L2) [right of =H, below of =H, node distance = 2cm ] {$ (s^{0,\Omega}_{\delta,\delta}f)_\mu$};
 \draw[->] (H) to node {$id_\mu^*$} (L);
 \draw[->] (H) to node [swap] {$id^2$} (L2);
 \draw[->] (L2) to node [swap] {$id^1$} (L);
 \end{tikzpicture}
\end{center}
Using property $(c)$ of the $s$-numbers we conclude
\[ 
x_n(id_{\mu}^*)\leq \| \, id^1 \, \| \, x_n(id^2) \, .
\]
By Corollary \ref{ba2-1}, we have
$$ \| id^1 \| \lesssim \mu^{(d-1)(\frac{1}{2}-\frac{1}{\delta})} \, . $$
From Lemma \ref{ba1} (ii), we derive
$$ x_n(id^2)\lesssim 2^{\mu( -t+\frac{1}{p_0} -\frac{1}{\delta})}\, x_n(id_{p_0,\delta}^{D_\mu})\, ,$$
taking into account property $(c)$ of the $s$-numbers and the commutative diagram
\beqq
\begin{CD}
2^{\mu(t-\frac{1}{p_0})}\, (\ell_{p_0}^{D_\mu}) @ > id^3 >> \ell_{p_0}^{D_\mu}\\
@V id^2 VV @VV id_{p_0,\delta}^{D_\mu} V\\
2^{-\frac{\mu}{\delta}}\, (\ell_{\delta}^{D_\mu}) @ << id^4 < \ell_{\delta}^{D_\mu} \, ,
\end{CD}
\eeqq
i.e., $id^2= id^4 \circ id_{p_0,\delta}^{D_\mu} \circ id^3 $,
\[
\|\, id^3\, \|= 2^{-\mu(t-\frac{1}{p_0})}\, \qquad \mbox{and}\qquad 
\|\, id^4\, \|= 2^{\frac{\mu}{\delta}}\,.
\]
Altogether this implies
\beqq
 x_n(id_{\mu}^*)\lesssim \mu^{(d-1)(\frac{1}{2}-\frac{1}{\delta})} \, 2^{\mu( -t+\frac{1}{p_0} -\frac{1}{\delta})} \, 
x_n(id_{p_0,\delta}^{D_\mu})
\, .
\eeqq
{\em Step 2.}
Now we turn to the estimate from below. We define
$ \gamma :=\min(p,2)$ and use the following commutative diagram

\tikzset{node distance=4cm, auto}

\begin{center}
\begin{tikzpicture}
 \node (H) {$(s^{t,\Omega}_{p_0,p_0}b)_\mu $};
 \node (L) [right of =H] {$(s^{0,\Omega}_{\gamma,\gamma}f)_\mu$};
 \node (L2) [right of =H, below of =H, node distance = 2cm ] {$(s^{0,\Omega}_{p,2}f)_\mu$};
 \draw[->] (H) to node {$id^2$} (L);
 \draw[->] (H) to node [swap] {$id_\mu^*$} (L2);
 \draw[->] (L2) to node [swap] {$id^1$} (L);
 \end{tikzpicture}
\end{center}

\noindent
This time we have $ x_n(id_2)\leq \|\, id^1\, \|\, x_n(id_{\mu}^*) $ and by Corollary \ref{ba2-1}, we get
$$ \| \, id^1\, \| \lesssim \mu^{(d-1)(\frac{1}{\gamma}-\frac{1}{2})}\, . $$
Similarly as in Step 1 Lemma \ref{ba1} (ii) yields
\[ 
x_n(id^2)\gtrsim 2^{\mu( -t+\frac{1}{p_0} -\frac{1}{\gamma})} \, x_n (id_{p_0,\gamma}^{D_{\mu}})\, . 
\]
Inserting this in our previous estimate we find
\beqq
x_n(id_{\mu}^*)\gtrsim \mu^{(d-1)(\frac{1}{2}-\frac{1}{\gamma})} \, 2^{\mu( -t+\frac{1}{p_0} -\frac{1}{\gamma})}
\, x_n (id_{p_0,\gamma}^{D_{\mu}})\, .
\eeqq
The proof is complete.
\end{proof}

\begin{proposition}
Let $0<p_0\leq \infty$ and $t\in \mathbb{R}$. Then we have the following assertions.
\begin{enumerate}
\item If $0< p\leq 2$, then 
\begin{equation}\label{case3}
 x_n(id_{\mu}^*)\lesssim 2^{\mu(-t+\frac{1}{p_0}-\frac{1}{p})}\, x_n(id_{p_0,p}^{D_\mu})\, .
\end{equation}
\item If $2\leq p<\infty$, then 
\begin{equation}\label{case4}
2^{\mu( -t+\frac{1}{p_0} -\frac{1}{p})}\, x_n(id_{p_0,p}^{D_\mu}) \lesssim x_n(id_{\mu}^*)\, .
\end{equation}
\end{enumerate}
\end{proposition}

\begin{proof} {\em Step 1.} Proof of (i).
We consider the following diagram

\tikzset{node distance=4cm, auto}

\begin{center}
\begin{tikzpicture}
 \node (H) {$(s^{t,\Omega}_{p_0,p_0}b)_\mu$};
 \node (L) [right of =H] {$(s^{0,\Omega}_{p,2}f)_\mu$};
 \node (L2) [right of =H, below of =H, node distance = 2cm ] {$ (s^{0,\Omega}_{p,p}f)_\mu $};
 \draw[->] (H) to node {$id_\mu^*$} (L);
 \draw[->] (H) to node [swap] {$id^2$} (L2);
 \draw[->] (L2) to node [swap]{$id^1$} (L);
 \end{tikzpicture}
\end{center}

%
\noindent
This implies $ x_n(id_{\mu}^*)\leq \| \, id^1 \, \| \, x_n(id^2) $.
Corollary \ref{ba2-1} yields $ \| id^1 \| \lesssim 1$
and from Lemma \ref{ba1} we derive
$$ x_n(id^2)\lesssim 2^{\mu( -t+\frac{1}{p_0} -\frac{1}{p})}\, x_n(id_{p_0,p}^{D_\mu}) \, .$$
Altogether we have found
$$
 x_n(id_{\mu}^*)\lesssim 2^{\mu( -t+\frac{1}{p_0} -\frac{1}{p})}\, x_n(id_{p_0,p}^{D_\mu})\, .
$$
{\em Step 2.} Proof of (ii). We use the following diagram

\tikzset{node distance=4cm, auto}

\begin{center}
\begin{tikzpicture}
 \node (H) {$(s^{t,\Omega}_{p_0,p_0}b)_\mu $};
 \node (L) [right of =H] {$(s^{0,\Omega}_{p,p}f)_\mu $};
 \node (L2) [right of =H, below of =H, node distance = 2cm ] {$(s^{0,\Omega}_{p,2}f)_\mu$};
 \draw[->] (H) to node {$id^2$} (L);
 \draw[->] (H) to node [swap] {$id^*_\mu$} (L2);
 \draw[->] (L2) to node [swap]{$id^1$} (L);
 \end{tikzpicture}
\end{center}

\noindent
Because of $ x_n(id^2)\leq \| \, id^1\, \| \, x_n(id_{\mu}^*) $, $ \| \, id^1\, \| \lesssim 1$, see Corollary \ref{ba2-1}, 
and 
$$ x_n(id^2)\gtrsim 2^{\mu( -t+\frac{1}{p_0} -\frac{1}{p})} \, x_n(id_{p_0,p}^{D_\mu}) \, , $$
(see Lemma \ref{ba1}), we obtain
$$
x_n(id_{\mu}^*)\gtrsim 2^{\mu( -t+\frac{1}{p_0} -\frac{1}{p})}\, x_n(id_{p_0,p}^{D_\mu})\, . 
$$
The proof is complete.
\end{proof}

We need a few more results of the above type.

\begin{lemma}
Let $0<p_0,p<\infty$ and $0<\epsilon<p$. Then
\begin{equation}\label{case5}
x_n(id_{\mu}^*)\lesssim 2^{\mu(-t+\frac{1}{p_0}-\frac{1}{p})} \, x_n(id_{p_0,p-\epsilon}^{D_{\mu}})\, .
\end{equation}
\end{lemma}

\begin{proof}
We consider the following diagram

\tikzset{node distance=4cm, auto}

\begin{center}
\begin{tikzpicture}
 \node (H) {$(s^{t,\Omega}_{p_0,p_0}b)_\mu $};
 \node (L) [right of =H] {$(s^{0,\Omega}_{p,2}f)_\mu$};
 \node (L2) [right of =H, below of =H, node distance = 2cm ] {$ (s^{r,\Omega}_{p-\epsilon,p-\epsilon}f)_\mu$};
 \draw[->] (H) to node {$id^*_\mu$} (L);
 \draw[->] (H) to node [swap] {$id^2$} (L2);
 \draw[->] (L2) to node [swap] {$id^1$} (L);
 \end{tikzpicture}
\end{center}

\noindent
Clearly, $x_n(id_{\mu}^* )\leq \| \, id^1 \, \| \, x_n(id^{2})$ and by Lemma \ref{ba3} we have
$$\| \, id^1\, \|\lesssim 2^{\mu(-r+\frac{1}{p-\epsilon}-\frac{1}{p})} \, .$$
Further we know
\[
x_n(id^2) = 2^{\mu(r-\frac{1}{p-\epsilon}-t+\frac{1}{p_0})} \, x_n(id_{p_0,p-\epsilon}^{D_{\mu}})\, .
\]
Inserting the previous inequality in this identity we obtain \eqref{case5}. 
\end{proof}

\begin{lemma}
For all $\mu \in \N_0$ and all $n \in \N$ we have 
\begin{equation}\label{eqlow1}
x_n(id_{\mu}^*)\leq x_n(id^*) \, .
\end{equation}
\end{lemma}

\begin{proof}
We consider the following diagram
\beqq
\begin{CD}
s^{t,\Omega}_{p_0,p_0}b @ > id^* >> s^{0,\Omega}_{p,2}f\\
@A id^1 AA @VV id^2 V\\
(s^{t,\Omega}_{p_0,p_0}b)_\mu @ > id_{\mu}^* >> (s^{0,\Omega}_{p,2}f)_\mu \, .
\end{CD}
\eeqq
Here $id^1$ is the canonical embedding and $id^2$ is the canonical projection.
Since $id_{\mu}^* = id_2 \circ id^* \circ id_1$ the property $(c)$ of the $s$-numbers yields
$$
x_n(id_{\mu}^*) \leq \| \, id^1\, \| \, \|\, id^2 \, \| \, x_n(id^*) = x_n(id^*)\, .
$$
This completes the proof.
\end{proof}


\subsection{Weyl numbers of embeddings of sequence spaces related to spaces of dominating mixed smoothness - results}
\label{sequence3}


Now we are in position to deal with the Weyl numbers of $ id^* :\ s_{p_0,p_0}^{t,\Omega}b\to s_{p,2}^{0,\Omega}f$.
We have to continue with the proof already started in \eqref{ws-15}-\eqref{sum1}.
Therefore we need to distinguish several cases. Always the positions of $p_0$ and $p$ relative to $2$ are of importance.


\subsubsection{The case $0<p_0\leq 2\leq p< \infty$}


\begin{theorem}\label{the1-1} Let $0< p_0\leq 2\leq p< \infty $ and $t>\frac{1}{p_0}-\frac{1}{p}$. Then
\[
x_n(id^*) \asymp n^{-t+\frac{1}{2}-\frac{1}{p}} (\log n)^{(d-1)(t+\frac{1}{p}-\frac{1}{p_0})}\, \, , \qquad n\geq 2\, .
\]
\end{theorem}

\begin{proof}
{\em Step 1.} Estimate from below.
Since $p\geq 2$, from (\ref{eqlow1}) and (\ref{case4}) we derive
$$ x_n(id^*) \gtrsim 2^{\mu(-t+\frac{1}{p_0}-\frac{1}{p})}\, x_n(id_{p_0,p}^{D_\mu})\, .$$
Next we choose $n=[\frac{D_{\mu}}{2}]$ (here $[\, x\, ]$ denotes the integer part of $x\in \re$). Then from 
property (a) in Appendix A we get
\[
x_n(id_{p_0,p}^{D_\mu})\gtrsim (D_{\mu})^{\frac{1}{2}-\frac{1}{p_0}} \asymp (2^{\mu}\, \mu^{d-1})^{\frac{1}{2}-\frac{1}{p_0}}\, ,
\]
which implies
\[
 x_n(id^*) \gtrsim 2^{\mu( -t+\frac{1}{2} -\frac{1}{p})}\, \mu^{(d-1)(\frac{1}{2}-\frac{1}{p_0})}\, .
\]
Because of $ 2^{\mu}\asymp \frac{n}{\log^{d-1}n} $
we conclude
$$x_n(id^*) \gtrsim n^{ -t+\frac{1}{2} -\frac{1}{p}}(\log n)^{(d-1)( t-\frac{1}{p_0} +\frac{1}{p})} \, .$$
{\em Step 2.} Estimate from above. Let $L,J$ and $\alpha$ as in \eqref{ws-15}-\eqref{ws-17}. 
By our assumptions we obviously have
$$2^{-\alpha L}L^{(d-1)(\frac{1}{2}-\frac{1}{p_0})_+} = 2^{L(-t+\frac{1}{p_0}-\frac{1}{p})}\, . $$
For given $J$ we choose $L>J$ large enough such that
\be\label{ws-21}
2^{L(-t+\frac{1}{p_0}-\frac{1}{p})} \lesssim J^{(d-1)(\frac{1}{2}-\frac{1}{p_0})}\, 2^{J(-t+\frac{1}{2}-\frac{1}{p})} \, .
\ee
For the sum in (\ref{sum1}), we define 
\[ 
n_{\mu}: = [D_{\mu}\, 2^{(J-\mu)\lambda}] \leq \dfrac{D_{\mu}}{2},\qquad J+1\leq \mu\leq L\, , 
\]
where $\lambda>1$ is at our disposal. We choose $\lambda$ such that
\begin{equation}\label{cond1}
t-\dfrac{1}{2}+\dfrac{1}{p}>\lambda\bigg(\dfrac{1}{p_0}-\dfrac{1}{2}\bigg)
\end{equation}
which is always possible under the given restrictions.
Then 
\be\label{ws-22}
\sum_{\mu=J+1}^{L}n_{\mu}\asymp J^{d-1}2^J
\ee
follows. If $p>2$, we choose $\epsilon>0$ such that $2\leq p-\epsilon$. From (\ref{case5}) we obtain
$$
x_{n_{\mu}}(id_{\mu}^*) \lesssim 2^{\mu(-t+\frac{1}{p_0}-\frac{1}{p})}\, x_{n_{\mu}}(id_{p_0,p-\epsilon}^{D_\mu}).$$
If $p=2$, then (\ref{case2}) implies
$$x_{n_{\mu}}(id_{\mu}^*) \lesssim 2^{\mu(-t+\frac{1}{p_0}-\frac{1}{2})}\, x_{n_{\mu}}(id_{p_0,2}^{D_\mu})\, .$$
Employing property (a) in Appendix A we obtain
\beqq
x_{n_{\mu}}(id_{\mu}^*)& \lesssim & 2^{\mu(- t+\frac{1}{p_0}-\frac{1}{p})}\, \Big(\mu^{d-1} \, 2^{\mu} \, 2^{(J-\mu)\lambda}\Big)^{\frac{1}{2}-\frac{1}{p_0}}
\\
& = & \mu^{(d-1)(\frac{1}{2}-\frac{1}{p_0})}\, 2^{\mu(- t+\frac{1}{2}-\frac{1}{p})} \, 2^{(J-\mu)\lambda(\frac{1}{2}-\frac{1}{p_0})}\, .
\eeqq
Our special choice of $\lambda$ in \eqref{cond1} yields
\be\label{ws-24}
\sum_{\mu=J+1}^{L} x^{\rho}_{n_{\mu}}(id_{\mu}^*) \lesssim J^{(d-1)\rho(\frac{1}{2}-\frac{1}{p_0})}\, 2^{J\rho(- t+\frac{1}{2}-\frac{1}{p})}\, .
\ee
Inserting \eqref{ws-21} and \eqref{ws-24} into \eqref{sum1} leads to
\[
x^{\rho}_n(id^*) \lesssim J^{(d-1)\rho(\frac{1}{2}-\frac{1}{p_0})}\, 2^{J\rho(- t+\frac{1}{2}-\frac{1}{p})} \, .
\]
Notice
\[
n = 1 + \sum_{\mu=0}^L (n_\mu -1) = 1+ \sum_{\mu=0}^J D_\mu + \sum_{\mu=J+1}^L (
[ D_{\mu}\, 2^{(J-\mu)\lambda}]-1)\asymp J^{d-1}\, 2^J\, ,
\]
see \eqref{ws-18}, \eqref{ws-23} and \eqref{ws-22}.
Hence, our proof works for a certain subsequence $(n_J)_{J=1}^\infty$ of the natural numbers. More exactly, with
\[
n_J := 1 + \sum_{\mu=0}^J D_\mu + \sum_{\mu=J+1}^L (
[ D_{\mu}\, 2^{(J-\mu)\lambda}]-1)\, , \qquad J \in \N\, , 
\] 
and $L=L(J)$ chosen as the minimal admissible value in \eqref{ws-21}
we find
\[
x_{n_J}(id^*) \lesssim J^{(d-1)(\frac{1}{2}-\frac{1}{p_0})}\, 2^{J(- t+\frac{1}{2}-\frac{1}{p})} \, .
\]
We already know
\[
A\, J^{d-1}\, 2^J\le n_J \le B \, J^{d-1}\, 2^J\, , \qquad J \in \N\, ,
\] 
for suitable $A,B>0$. Without loss of generality we assume $B \in \N$.
 Then we conclude from the monotonicity of the Weyl numbers
\[
x_{B \, J^{d-1}\, 2^J}(id^*) \lesssim \log (B \, J^{d-1}\, 2^J)^{(d-1)(\frac{1}{2}-\frac{1}{p_0})}\, 
\Big(\frac{B \, J^{d-1}\, 2^J}{\log^{d-1} (B \, J^{d-1}\, 2^J)}\Big)^{- t+\frac{1}{2}-\frac{1}{p}} \, .
\]
Employing one more times the monotonicity of the Weyl numbers and in addition its polynomial behaviour we can switch 
from the subsequence $(B \, J^{d-1}\, 2^J)_J$ to $n\in \N$ in this formula by possibly changing the constant behind $\lesssim$.
This finishes our proof.
\end{proof}


\subsubsection{The case $2\leq p_0 \le p< \infty$}


\begin{theorem}\label{the2-1}
Let $2\leq p_0 \le p<\infty$ and $t>\frac{1}{p_0}-\frac{1}{p}$. Then
\[
x_n(id^*)\asymp n^{-t+\frac{1}{p_0}-\frac{1}{p}} (\log n)^{(d-1)(t+\frac{1}{p}-\frac{1}{p_0})}\, , \qquad n \geq 2\, .
\]
\end{theorem}

\begin{proof}
{\em Step 1.} Estimate from below. We apply the same arguments as in proof of the previous theorem.
However, notice that $x_n (id^{D_\mu}_{p_0,p})$ has a different behaviour, see property (a) in Appendix A.
With $n=[{D_{\mu}}/{2}]$ and $x_n(id_{p_0,p}^{D_\mu})\gtrsim 1$ we conclude that
\[
x_n(id^*) \gtrsim 2^{\mu(-t+\frac{1}{p_0}-\frac{1}{p})}\, x_n(id_{p_0,p}^{D_\mu})\gtrsim 2^{\mu(-t+\frac{1}{p_0}-\frac{1}{p})}\, ,
\]
see \eqref{eqlow1} and \eqref{case4}. Because of $2^{\mu}\asymp \frac{n}{\log^{d-1}n}$ this results in the estimate
\[
x_n(id^*) \gtrsim n^{ -t+\frac{1}{p_0} -\frac{1}{p}}(\log n)^{(d-1)( t-\frac{1}{p_0} +\frac{1}{p})}\, .
\]
{\em Step 2.} Estimate from above. For $J\in \mathbb{N}$ and $\lambda\in s^{t,\Omega}_{p_0,p_0}b$ we put
\[
S_J\lambda :=\sum_{\mu=0}^{J}\sum_{|\bar{\nu}|_1=\mu}\sum_{\bar{m}\in A_{\bar{\nu}}^{\Omega}}\lambda_{\bar{\nu},\bar{m}}e^{\bar{\nu},\bar{m}}\, , 
\]
where $\{e^{\bar{\nu},\bar{m}}, \bar{\nu}\in \mathbb{N}_0^d,\ \bar{m}\in A_{\bar{\nu}}^{\Omega}\}$ is the canonical orthonormal basic of 
$s^{0,\Omega}_{2,2}b$. Obviously
\[
\| id^*-S_J: s^{t,\Omega}_{p_0,p_0}b \to s^{0,\Omega}_{p,2}f\|
\leq \sum_{\mu=J+1}^{\infty}\| id^*_{\mu}: (s^{t,\Omega}_{p_0,p_0}b)_{\mu} \to (s^{0,\Omega}_{p,2}f)_{\mu}\|.
\]
Using Lem. \ref{ba3} and $(s^{t,\Omega}_{p_0,p_0}b)_{\mu} = (s^{t,\Omega}_{p_0,p_0}f)_{\mu}$ 
we get
\[
\| id^*-S_J: s^{t,\Omega}_{p_0,p_0}b \to s^{0,\Omega}_{p,2}f\|\leq \sum_{\mu=J+1}^{\infty}2^{-\mu(t-\frac{1}{p_0}+\frac{1}{p})}
\lesssim 2^{-J(t-\frac{1}{p_0}+\frac{1}{p})} \, .
\] 
Because of $\text{rank}(S_J)\asymp 2^JJ^{d-1}$ we conclude in case $n=2^JJ^{d-1}$ that 
\[
a_n( id^*)\lesssim 2^{-J(t-\frac{1}{p_0}+\frac{1}{p})}\, .
\]
Since $x_n \leq a_n$
we can complete the proof of the estimate from above by arguing as at the end of the proof of Thm. \ref{the1-1}.
\end{proof}


\subsubsection{The case $2\leq p < p_0\leq \infty$}


\begin{theorem}\label{the4-1}
Let $ 2\leq p < p_0 \leq \infty$ and $t>\frac{ {1}/{p}-{1}/{p_0}}{{p_0}/{2}-1}$. Then
\[ 
x_n(id^*) \asymp n^{-t+\frac{1}{p_0}-\frac{1}{p}} \, (\log n)^{(d-1)(t+\frac{1}{p}-\frac{1}{p_0})} \, , \qquad n \geq 2\, .
\]
\end{theorem}

\begin{proof}
{\em Step 1.} Estimate from below.
Because of $p>2$, \eqref{eqlow1} and \eqref{case4} imply
\[
x_n(id^*) \gtrsim 2^{\mu(-t+\frac{1}{p_0}-\frac{1}{p})}\, x_n(id_{p_0,p}^{D_\mu})\, .
\]
We choose $n=[D_{\mu}/2]$. Then property (b)(part(iii)) in Appendix A yields
$x_n(id_{p_0,p}^{D_{\mu}} )\gtrsim 1$.
Hence
\[ 
x_n(id^*) \gtrsim 2^{\mu(-t +\frac{1}{p_0} -\frac{1}{p})} \, .
\]
Because of $2^{\mu}\asymp \frac{n}{\log^{d-1}n}$ this implies the desired estimate.
\\
{\em Step 2.} Estimate from above. Since $2\leq p< p_0$ we obtain
\[
2^{-\alpha L}\, L^{(d-1)(\frac{1}{2}-\frac{1}{p_0})_+} = 2^{-Lt}\, L^{(d-1)(\frac{1}{2} -\frac{1}{p_0})}\, .
\]
For given $J$ we choose $L$ large enough such that
\begin{equation}\label{ws-25}
2^{-Lt}L^{(d-1)(\frac{1}{2}-\frac{1}{p_0})}\leq 2^{-\gamma Jt}
\end{equation}
for some $\gamma > 1$ (to be chosen later on).
We define 
\[
n_{\mu} := [D_{\mu}\, 2^{(J-\mu)\beta}]\leq D_{\mu},\ \ J+1\leq \mu\leq L,
\]
where the parameter $\beta>1$ will be also chosen later on.
Hence
\[
\sum_{\mu=J+1}^{L} \, n_{\mu}\asymp 2^{J}\, J^{d-1}\, .
\]
The restriction $t>\frac{ {1}/{p}- {1}/{p_0}}{ {p_0}/{2}-1} $ implies
\[
-t + \frac{1}{p_0}-\frac{1}{p} + \frac{ {1}/{p}- {1}/{p_0}}{1- {2}/{p_0}} < 0\, .
\]
If $p>2$ we choose $\epsilon >0$ such that $ 2\leq p-\epsilon$ and
\begin{equation}\label{cond4}
-t+\frac{1}{p_0}-\frac{1}{p} + \frac{\frac{1}{p-\epsilon}-\frac{1}{p_0}}{1-\frac{2}{p_0}} < 0\, .
\end{equation}
In this situation we derive from property (b)(part(i)) in Appendix A
\[
x_{n_{\mu}}(id_{p_0,p-\epsilon}^{D_{\mu}})\lesssim 
\bigg(\dfrac{D_{\mu}}{n_{\mu}}\bigg)^{\frac{1}{r}} \asymp 2^{-\frac{(J-\mu)\beta}{r}},\qquad \dfrac{1}{r}:=
\dfrac{\frac{1}{p-\epsilon}-\frac{1}{p_0}}{1-\frac{2}{p_0}}\, .
\]
The estimate \eqref{case5} guarantees
\beq\label{ws-29-1}
\sum_{\mu=J+1}^{L}x^\rho_{n_{\mu}}(id_{\mu}^*) \lesssim  \sum_{\mu=J+1}^{L} 
2^{\mu\rho(-t+\frac{1}{p_0}-\frac{1}{p})}\, x^\rho_{n_{\mu}}(id_{p_0,p-\epsilon}^{D_{\mu}})\, .
\eeq
In case $p=2$, again property (b)(part(i)) in Appendix A yields
\[
x_{n_{\mu}}(id_{p_0,2}^{D_{\mu}})\lesssim 
\bigg(\dfrac{D_{\mu}}{n_{\mu}}\bigg)^{\frac{1}{2}} \asymp 2^{-\frac{(J-\mu)\beta}{r}},\qquad \dfrac{1}{r}:=
\dfrac{1}{2}\, .
\]
From (\ref{case3}) we obtain
\beq\label{ws-29-2}
\sum_{\mu=J+1}^{L}x^\rho_{n_{\mu}}(id_{\mu}^*) \lesssim \sum_{\mu=J+1}^{L} 
2^{\mu\rho(-t+\frac{1}{p_0}-\frac{1}{2})}\, x^\rho_{n_{\mu}}(id_{p_0,2}^{D_{\mu}})\, .
\eeq
Now (\ref{ws-29-1}) and (\ref{ws-29-2}) yield
\beqq
\sum_{\mu=J+1}^{L}x^\rho_{n_{\mu}}(id_{\mu}^*) & \lesssim & \sum_{\mu=J+1}^{L} 2^{\mu\rho(-t+\frac{1}{p_0}-\frac{1}{p})}\, 
2^{-\frac{(J-\mu)\beta\rho}{r}} 
\\
& = & \sum_{\mu=J+1}^{L}\ 2^{\mu\rho(-t+\frac{1}{p_0}-\frac{1}{p}+\frac{\beta}{r})}\, 2^{-\frac{J\beta\rho}{r}}\, .
\eeqq
The condition (\ref{cond4}) can be rewritten as
 \[
-t+\frac{1}{p_0}-\frac{1}{p}+\frac{1}{r}<0 \, .
 \]
Now we choose $\beta>1$ such that $-t+\frac{1}{p_0}-\frac{1}{p}+\frac{\beta}{r}<0$. 
Then 
\[
\sum_{\mu=J+1}^{L}x^\rho_{n_{\mu}}(id_{\mu}^*)\lesssim 2^{J\rho(-t+\frac{1}{p_0}-\frac{1}{p}+\frac{\beta}{r})}
\, 2^{-\frac{J\beta\rho}{r}} = 2^{J\rho(-t+\frac{1}{p_0}-\frac{1}{p})}
\]
follows.
Inserting this and \eqref{ws-25} into \eqref{sum1} we find
\[
x_n(id^*) \lesssim \Big( 2^{J\rho(-t+\frac{1}{p_0}-\frac{1}{p})}+ 2^{-\gamma Jt\rho}\Big)\, .
\]
Choosing
\[
\gamma :=\dfrac{-t+\frac{1}{p_0}-\frac{1}{p}}{-t} > 1
\]
 then we conclude
\[
x_{n}(id^*)\lesssim 2^{J(-t+\frac{1}{p_0}-\frac{1}{p})}
\]
and this is enough to prove the estimate from above, compare with
the end of the proof of Thm. \ref{the1-1}.
\end{proof}

\begin{theorem}\label{the4-2}
Let $2\leq p< p_0 < \infty$ and $0<t<\frac{ {1}/{p}- {1}/{p_0}}{ {p_0}/{2}-1}$. Then
\[
x_n(id^*)\asymp n^{-\frac{tp_0}{2}}(\log n)^{(d-1)(t+\frac{1}{2}-\frac{1}{p_0})}\, , \qquad n \geq 2\, .
\]
\end{theorem}

\begin{proof}
{\em Step 1.} Estimate from below. From \eqref{case2} and (\ref{eqlow1}) we derive
$$
2^{\mu(-t+\frac{1}{p_0}-\frac{1}{2})}\, x_n(id_{p_0,2}^{D_{\mu}} )\lesssim x_n(id^*)\, .
$$
Now we choose $n=[D_{\mu}^{\frac{2}{p_0}}]$. Then it follows from property (b)(part(ii)) in Appendix A that
\[
x_n(id_{p_0,2}^{D_{\mu}} )\gtrsim D_{\mu}^{\frac{1}{2}-\frac{1}{p_0}} \gtrsim (\mu^{d-1}2^{\mu})^{\frac{1}{2}-\frac{1}{p_0}}\, .
\]
This implies
\[
x_n(id^*) \gtrsim \mu^{(d-1)(\frac{1}{2}-\frac{1}{p_0})} \, 2^{-t\mu} \, .
\]
Rewriting the right-hand side in dependence on $n$ we obtain
\[
x_n(id^*) \gtrsim n^{-\frac{tp_0}{2}}(\log n)^{(d-1)(t+\frac{1}{2}-\frac{1}{p_0})}\, .
\]
{\em Step 2.} Estimate from above.
Since $2\leq p< p_0$ we have 
\[
2^{-\alpha L }\, L^{(d-1)(\frac{1}{2}-\frac{1}{p_0})_+}= 2^{-t L}\, L^{(d-1)(\frac{1}{2}-\frac{1}{p_0})}\, . 
\]
For fixed $J\in \N$ we choose 
\[
L:=\Big[\frac{p_0}{2} J+(d-1)(\frac{p_0}{2}-1)\log J\Big] \, .
\] 
Hence
\[ 
2^{-Lt} = 2^{-t([\frac{p_0}{2} J+(d-1)(\frac{p_0}{2}-1)\log J])}\asymp 2^{-\frac{p_0}{2} Jt}\, J^{(d-1)(t-\frac{tp_0}{2})}
\]
and
\[
L^{{(d-1)(\frac{1}{2}-\frac{1}{p_0})}}= \Big(\Big[\frac{p_0}{2} J+(d-1)(\frac{p_0}{2}-1)\log J\Big]\Big)^{(d-1)(\frac{1}{2}-\frac{1}{p_0})}
\lesssim J^{(d-1)(\frac{1}{2}-\frac{1}{p_0})}\, .
\]
This results in the estimate
\be\label{ws-26}
2^{-Lt}\, L^{(d-1)(\frac{1}{2}-\frac{1}{p_0})}\lesssim 
2^{-\frac{p_0}{2} Jt}\, J^{(d-1)(t-\frac{tp_0}{2}+\frac{1}{2}-\frac{1}{p_0})} \, .
\ee
We define
\[
n_{\mu}:= \big[D_{\mu}\, 2^{\{(\mu-L)\beta+J-\mu \}}\big]\leq D_{\mu}\, , \qquad J+1 \le \mu \le L\, ,
\]
where $\beta >0$ will be fixed later on. Consequently
\be\label{ws-27}
\sum_{\mu=J+1}^{L} n_{\mu} \lesssim 2^{J}\, J^{d-1}\, .
\ee
Employing property (b)(part(i)) in Appendix A we get
\be\label{ws-29}
x_{n_{\mu}}(id_{p_0,p}^{D_{\mu}})\lesssim \Big(\dfrac{D_{\mu}}{n_{\mu}}\Big)^{\frac{1}{r}} \lesssim 2^{-\frac{(\mu-L)\beta+J-\mu}{r}}\, , 
\qquad \frac 1r := \frac{ {1}/{p} -  {1}/{p_0}}{1-  {2}/{p_0}}
\, .
\ee
We continue by applying \eqref{case2}
\beqq
\sum_{\mu=J+1}^{L} x^\rho_{n_{\mu}}(id_{\mu}^*) & \lesssim & 
 \sum_{\mu=J+1}^{L} \mu^{(d-1)\rho(\frac{1}{2}-\frac{1}{p})} \, 2^{\mu\rho(-t+\frac{1}{p_0}-\frac{1}{p})} \, 
x^\rho_{n_{\mu}}(id_{p_0,p}^{D_{\mu}})
\\
& \lesssim & 
\sum_{\mu=J+1}^{L} \mu^{(d-1)\rho(\frac{1}{2}-\frac{1}{p})}\, 2^{\mu\rho(-t+\frac{1}{p_0}-\frac{1}{p})}
\, 2^{-\frac{\{(\mu-L)\beta+J-\mu \}\rho}{r}}
\\
& = & \sum_{\mu=J+1}^{L} \mu^{(d-1)\rho(\frac{1}{2}-\frac{1}{p})}\, 
2^{\mu\rho(-t+\frac{1}{p_0}-\frac{1}{p}+\frac{1}{r}-\frac{\beta}{r})}\, 2^{\frac{(L\beta-J )\rho}{r}}\, .
\eeqq
Because of 
\[ 
t < \frac{ {1}/{p}- {1}/{p_0}}{ {p_0}/{2}-1} \qquad \Longleftrightarrow \qquad -t+\frac{1}{p_0}-\frac{1}{p}+\frac{1}{r}>0
 \]
we can choose $\beta>0$ such that $-t+\frac{1}{p_0}-\frac{1}{p}+\frac{1}{r}-\frac{\beta}{r}>0$.
Then 
\beq\label{ws-28}
\sum_{\mu=J+1}^{L}x^\rho_{n_{\mu}}(id_{\mu}^*)& \lesssim &
L^{(d-1)\rho(\frac{1}{2}-\frac{1}{p})}\, 2^{L\rho(-t+\frac{1}{p_0}-\frac{1}{p}+\frac{1}{r}-\frac{\beta}{r})}\, 
2^{\frac{(L\beta-J)\rho}{r}}
\nonumber
\\
&= & L^{(d-1)\rho(\frac{1}{2}-\frac{1}{p})} \, 2^{L\rho(-t+\frac{1}{p_0}-\frac{1}{p}+\frac{1}{r})}\, 2^{-\frac{J\rho}{r}}
\eeq
follows.
Inserting the definition of $L $ we conclude
\[
L^{(d-1)(\frac{1}{2}-\frac{1}{p})}\lesssim J^{(d-1)(\frac{1}{2}-\frac{1}{p})}
\]
and 
\beqq
2^{L(-t+\frac{1}{p_0}-\frac{1}{p}+\frac{1}{r})}\, 2^{-\frac{J}{r}}
& \lesssim & 
2^{\big[\frac{p_0}{2}J +(d-1)(\frac{p_0}{2}-1)\log J \big]\big[-t+\frac{1}{p_0}-\frac{1}{p}+\frac{1}{r}\big]}2^{-\frac{J}{r}}
\\
& \lesssim & 2^{-\frac{tp_0}{2}J} \, J^{(d-1)(\frac{p_0}{2}-1)(-t+\frac{1}{p_0}-\frac{1}{p}+\frac{1}{r})}
\\
&= & 2^{-\frac{tp_0}{2}J}\, J^{(d-1)(t-\frac{tp_0}{2}-\frac{1}{p_0}+\frac{1}{p})}\, .
\eeqq
Now \eqref{ws-28} yields
\[
\sum_{\mu=J+1}^{L}x^\rho_{n_{\mu}}(id_{\mu}^*)\lesssim J^{(d-1)\rho(t-\frac{tp_0}{2}-\frac{1}{p_0}+\frac{1}{2})}\, 2^{-\frac{tp_0}{2}J\rho}
\, .
\]
This, together with \eqref{ws-26}, has to be inserted into \eqref{sum1}
\[
x_{n}(id^*)\lesssim J^{(d-1)(t-\frac{tp_0}{2}-\frac{1}{p_0}+\frac{1}{2})}\, 2^{-\frac{tp_0}{2}J}\, .
\]
The same type of arguments as at the end of the proof of Thm. \ref{the1-1} complete the proof.
\end{proof}

\begin{remark}\rm
Without going into details we mention the following estimate for the limiting 
case $t=\frac{ {1}/{p}- {1}/{p_0}}{ {p_0}/{2}-1} $. For all $n\ge 2$ we have
\[
n^{-\frac{tp_0}{2}}(\log n)^{(d-1)(t+\frac{1}{2}-\frac{1}{p_0})} \lesssim 
x_n(id^*) \lesssim n^{-\frac{tp_0}{2}}(\log n)^{(d-1)(t+\frac{1}{2}-\frac{1}{p_0})}(\log n)^{\frac{1}{r}+\frac{1}{\rho}}\, ,
\]
where $r$ is as in \eqref{ws-29} and $\rho=\min(1,p)$.
\end{remark}


\subsubsection{The case $0<p_0,p\leq 2$}


We need some preparations.

\begin{lemma}\label{the5-1}
Let $0<p_0,p\leq 2$ and $t>(\frac{1}{p_0}-\frac{1}{p})_+$. Then
\[
n^{-t}(\log n)^{(d-1)(t+\frac{1}{2}-\frac{1}{p_0})_+}\lesssim x_n(id^*)
\]
holds for all $n \ge 2$.
\end{lemma}

\begin{proof}
{\em Step 1. }
We consider the following commutative diagram
\[
\begin{CD}
s^{t,\Omega}_{p_0,p_0}b @ > id^* >> s^{0,\Omega}_{p,2}f \\
@A id^1 AA @VV id^2 V\\
2^{\mu(t-\frac{1}{p_0})}\ell_{p_0}^{A_{\mu}} @ > I_{\mu} >> 2^{\mu(0-\frac{1}{p})} \ell_{p}^{A_{\mu}} \,.
\end{CD}
\]
Here $A_{\mu}=|A_{\bar{\nu}}^{\Omega}|$ for some $\bar{\nu}$ with $|\bar{\nu}|_1=\mu$, $id^1$ is the canonical embedding, 
whereas $id^2$ is the canonical projection. From property $(c)$ of the $s$-numbers we derive 
\[
x_n(I_{\mu})=x_n(id^2\circ id^* \circ id^1)\leq \| id^1\| \, \|id^2\| \, x_n(id^*)=x_n(id^*)\, .
\]
Again the ideal property of the $s$-numbers guarantees
\[
x_n(I_{\mu})= 2^{\mu(-t+\frac{1}{p_0}-\frac{1}{p})} \, x_n(id_{p_0,p}^{A_{\mu}})\, .
\]
We choose $n=[A_{\mu}/2]$. Then property (a) in  Appendix A yields
\[
x_n(I_{\mu})\geq 2^{\mu(-t+\frac{1}{p_0}-\frac{1}{p})} \, x_n(id_{p_0,p}^{A_{\mu}})
\gtrsim 2^{\mu(-t+\frac{1}{p_0}-\frac{1}{p})} \, 2^{\mu(\frac{1}{p}-\frac{1}{p_0})} = 2^{-\mu t} \asymp n^{-t}\, ,
\]
which implies $x_n(id^*)\gtrsim n^{-t}$. This proves the lemma if $t+\frac{1}{2}-\frac{1}{p_0}\le 0$.
\\
{\em Step 2.} From (\ref{eqlow1}) and (\ref{case1}) we have
\[
x_n(id^*)\gtrsim \mu^{(d-1)(\frac{1}{2}-\frac{1}{p})}\, 2^{\mu(-t+\frac{1}{p_0}-\frac{1}{p})}\, x_n(id_{p_0,p}^{D_\mu})\, .
\]
We choose $n:=[D_{\mu}/2]$. Then property (a) in Appendix A leads to
\[
x_n(id_{p_0,p}^{D_{\mu}} )\gtrsim D_{\mu}^{\frac{1}{p}-\frac{1}{p_0}}\gtrsim (\mu^{d-1}2^{\mu})^{\frac{1}{p}-\frac{1}{p_0}}
\, ,
\]
which implies
\[ 
x_n(id^*) \gtrsim \mu^{(d-1)(\frac{1}{2}-\frac{1}{p_0})}\, 2^{-t\mu}\, .
\]
Because of  $ 2^{\mu}\asymp \frac{n}{\log^{d-1}n} $ this yields
\[
x_n(id^*) \gtrsim n^{-t}(\log n)^{(d-1)(t+\frac{1}{2}-\frac{1}{p_0})}\, .
\]
The proof is complete.
\end{proof}

\begin{lemma}\label{the5-2}
If $0<p_0, p\leq 2$ and $t>\frac{1}{p_0}-\frac{1}{2}$. Then
\[
x_n(id^*)\lesssim n^{-t}(\log n)^{(d-1)(t+\frac{1}{2}-\frac{1}{p_0})}
\]
holds for all $n \ge 2$.
\end{lemma}

\begin{proof} The restriction $t>\frac{1}{p_0}-\frac{1}{2}$ implies the following chain of continuous embeddings
\[
s_{p_0,p_0}^{t,\Omega}b\hookrightarrow s_{2,2}^{0,\Omega}f \hookrightarrow s_{p,2}^{0,\Omega}f\, .
\]
Now we consider the commutative diagram

\tikzset{node distance=4cm, auto}

\begin{center}
\begin{tikzpicture}
 \node (H) {$s^{t,\Omega}_{p_0,p_0}b$};
 \node (L) [right of =H] {$s^{0,\Omega}_{p,2}f $};
 \node (L2) [right of =H, below of =H, node distance = 2cm ] {$ s^{0,\Omega}_{2,2}f$};
 \draw[->] (H) to node {$id^*$} (L);
 \draw[->] (H) to node [swap] {$id^1$} (L2);
 \draw[->] (L2) to node [swap] {$id^2$} (L);
 \end{tikzpicture}
\end{center}
\noindent
The ideal property of the $s$-numbers  and Thm. \ref{the1-1} (applied with $p=2$) yield the claim.
\end{proof}

\begin{lemma}\label{the5-3}
Let $0<p \leq p_0< 2$ and $0<t<\frac{1}{p_0}-\frac{1}{2}$. Then 
$$x_n(id^*)\lesssim n^{-t} $$
holds for all $n\ge 1$.
\end{lemma}

\begin{proof}
For given $J \in \N$ we choose $ L:=J+(d-1)\, [\log J]$. Then
\be\label{ws-30}
2^{-L\alpha}\, L^{(d-1)(\frac{1}{2}-\frac{1}{p_0})_+} = 2^{-Lt} \asymp 2^{-tJ}\, J^{(d-1)(-t)}\, .
\ee
We define
\[
n_{\mu}:= \big[D_{\mu}\, 2^{(\mu-L)\beta+J-\mu}\big] \, , \qquad J+1 \le \mu \le L\, , 
\]
for some $\beta>0$. Then \eqref{ws-27} follows. Property (a) in Appendix A yields
\[
x_{n_{\mu}}(id_{p_0,2}^{D_{\mu}})\lesssim \Big(D_{\mu}2^{(\mu-L)\beta+J-\mu}\Big)^{\frac{1}{2}-\frac{1}{p_0}}\, .
\]
This, in connection with \eqref{case1}, leads to
\beqq
\sum_{\mu=J+1}^{L} x^{\rho}_{n_{\mu}}(id_{\mu}^*)& \lesssim & 
\sum_{\mu=J+1}^{L}2^{\mu{\rho}(- t+\frac{1}{p_0}-\frac{1}{2})} \Big(D_{\mu}\, 2^{(\mu-L)\beta+J-\mu}\Big)^{{\rho}(\frac{1}{2}-\frac{1}{p_0})}
\\
& \lesssim & \sum_{\mu=J+1}^{L} 2^{\mu{\rho}(- t+\frac{1}{p_0}-\frac{1}{2}+(\frac{1}{2}-\frac{1}{p_0})\beta)}\Big(\mu^{(d-1)}2^{-L\beta+J}\Big)^{\rho(\frac{1}{2}-\frac{1}{p_0})}\, .
\eeqq
Because of $t< \frac{1}{p_0}-\frac{1}{2}$ we can select $\beta>0$ such that
\[
-t+\frac{1}{p_0}-\frac{1}{2}+\big(\frac{1}{2}-\frac{1}{p_0}\big)\beta >0 \, .
\]
Consequently
\beq\label{ws-31}
\sum_{\mu=J+1}^{L}x^{\rho}_{n_{\mu}}(id_{\mu}^*) & \lesssim & 
2^{L{\rho}(- t+\frac{1}{p_0}-\frac{1}{2}+(\frac{1}{2}-\frac{1}{p_0})\beta)} \, \Big(L^{(d-1)} \, 2^{-L\beta+J}\Big)^{\rho(\frac{1}{2}-\frac{1}{p_0})} 
\nonumber
\\
& = & 2^{L\rho (- t+\frac{1}{p_0}- \frac 12)}\, \Big(L^{(d-1)}\, 2^{J}\Big)^{\rho(\frac{1}{2}-\frac{1}{p_0})}
\nonumber
\\
& \lesssim & 2^{L\rho(- t+\frac{1}{p_0}- \frac{1}2)}\, \Big(J^{(d-1)}\, 2^{J}\Big)^{\rho(\frac{1}{2}-\frac{1}{p_0})}
\nonumber
\\
& = & 2^{L\rho(- t)} \, 2^{L\rho(\frac{1}{p_0}- \frac 12)}\, \Big(J^{(d-1)}\, 2^{J}\, \Big)^{\rho(\frac{1}{2}-\frac{1}{p_0})}\, .
\eeq
Observe
\[
2^{L(\frac{1}{p_0}-\frac{1}{2})} \, \Big(J^{(d-1)} \, 2^{J}\Big)^{\frac{1}{2}-\frac{1}{p_0}} = 2^{(J+(d-1)[\log J])(\frac{1}{p_0}-\frac{1}{2})} 
\Big(J^{(d-1)}\, 2^{J}\Big)^{\frac{1}{2}-\frac{1}{p_0}}\asymp 1 \, .
\]
Replacing $L$ by $J+(d-1)[\log J] $ in \eqref{ws-31} we obtain
\[
\sum_{\mu=J+1}^{L} x^{\rho}_{n_{\mu}}(id_{\mu}^*)\lesssim 2^{-L\rho t} \lesssim (2^{J} \, J^{d-1})^{-\rho t} .
\]
This inequality, together with \eqref{ws-30}, yield
\[
x_{n_J}(id^*) \lesssim n_J^{-t}\, , 
\]
where 
\[
n_J := 1 + \sum_{\mu=0}^J D_\mu + \sum_{\mu=J+1}^L \Big(\big[D_{\mu}\, 2^{(\mu-L)\beta+J-\mu}\big] -1\Big)\, , \qquad J \in \N\, . 
\] 
Now we can continue as at the end of the proof of Thm. \ref{the1-1}.
\end{proof}

It remains to investigate the following situation: $0<p_0<p<2$ and $\frac{1}{p_0}-\frac{1}{p} < t <\frac{1}{p_0}-\frac{1}{2}$. 
The estimates of the Weyl numbers $x_n(id^*)$ from above will be the most complicated part within this paper.

\begin{lemma}\label{the5-4}
Let $0<p_0<p<2$ and $\frac{1}{p_0}-\frac{1}{p} < t <\frac{1}{p_0}-\frac{1}{2}$. Then
\[
x_n(id^*)\lesssim n^{-t}
\]
holds for all $n \ge 1$.
\end{lemma}

\begin{proof}
{\em Step 1.} We need to replace the decomposition of $id^*$ from \eqref{ws-15} by a more sophisticated one:
\beqq
id^* = \sum_{\mu=0}^{J} id_{\mu}^* + \sum_{\mu=J+1}^{L} id_{\mu}^* + \sum_{\mu=L+1}^{K} id_{\mu}^* + \sum_{\mu=K+1}^{\infty}id_{\mu}^*
\qquad \mbox{with}\quad J < L < K\, .
\eeqq
Here $J,L$ and $K$ will be chosen later on.
As in \eqref{ws-16} this decomposition results in the estimate
\be\label{ws-33}
x^\rho_n(id^*)\leq \sum_{\mu=0}^{J} x^\rho_{n_{\mu}}(id_{\mu}^*) + \sum_{\mu=J+1}^{L} x^\rho_{n_{\mu}}(id_{\mu}^*)
+ \sum_{\mu=L+1}^{K} x^\rho_{n_{\mu}}(id_{\mu}^*) + \sum_{\mu=K+1}^{\infty} \|id_{\mu}^*\|^\rho,\qquad \rho=\min(1,p)\, , 
\ee
where $n-1 = \sum_{\mu=0}^{K}(n_{\mu}-1)$.
Cor. \ref{ba2-1} yields
\[
 \|\, id_{\mu}^* \, \| \lesssim 2^{-\mu\alpha} \, \mu^{(d-1)(\frac{1}{2}-\frac{1}{p_0})_+} = 2^{\mu(-t+\frac{1}{p_0}-\frac{1}{p})} 
\]
and therefore
\[
\sum_{\mu=K+1}^{\infty} \|\, id_{\mu}^*\, \|^\rho \lesssim 2^{K\rho(-t+\frac{1}{p_0}-\frac{1}{p})} \, .
\]
As above we choose
\[
n_{\mu}:=D_{\mu}+1,\ \ \mu=0,1,....,J\, ,
\]
see \eqref{ws-18}. Hence
\beqq
\sum_{\mu=0}^{J} n_{\mu}\asymp J^{d-1} \, 2^J \qquad \mbox{and}\qquad \sum_{\mu=0}^{J}x^\rho_{n_{\mu}}(id_{\mu}^*)=0\, ,
\eeqq
see \eqref{ws-23} and \eqref{ws-19}.
Inserting this into \eqref{ws-33} we obtain
\begin{equation}\label{sum2}
x^\rho_n(id^*) \lesssim \sum_{\mu=J+1}^{L} x^\rho_{n_{\mu}}(id_{\mu}^*) + 
\sum_{\mu=L+1}^{K} x^\rho_{n_{\mu}}(id_{\mu}^*) + 2^{K\rho(-t+\frac{1}{p_0}-\frac{1}{p})}\, .
\end{equation}
{\em Step 2.} For given $J$ we choose $K$ large enough such that
\beqq
2^{K(-t+\frac{1}{p_0}-\frac{1}{p})} \leq 2^{-Jt}\, J^{(d-1)(-t)}\, .
\eeqq
Furthermore, we choose $L:= J+(d-1)[\log J]$ also in dependence on $J$. This implies
\beqq
 2^{-Lt}\asymp 2^{-tJ}\, J^{(d-1)(-t)}\, .
\eeqq
Now we fix our remaining degrees of freedom by defining 
\[
n_{\mu} : =
\begin{cases}
\big[D_{\mu} \, 2^{(\mu-L)\beta+J-\mu}\big] & \qquad \text{if}\quad J+1 \le \mu \leq L\, , \\
\big[J^{d-1}2^{J}\, 2^{(L-\mu)\gamma}\big] & \qquad \text{if}\quad L+1 \le \mu \leq K\, . 
\end{cases}
\]
Here $\beta, \, \gamma>0$ will be fixed later.
Since $\gamma >0$, applying \eqref{ws-27}, we have
\be\label{ws-37}
\sum_{\mu=J+1}^{K} n_{\mu}\asymp J^{d-1}2^J\, .
\ee
{\em Substep 2.1.} We estimate the first sum in \eqref{sum2}. 
Making use of the same arguments as in proof of Lemma \ref{the5-3}
we find
\be\label{ws-38}
\sum_{\mu=J+1}^{L} x^\rho_{n_{\mu}}(id_{\mu}^*) \lesssim 2^{-L\rho t} \lesssim 2^{-tJ\rho}\, J^{-(d-1)\rho t}\, .
\ee
{\em Substep 2.2.} Now we estimate the second sum in (\ref{sum2}). Therefore we consider the following splitting of $n_{\mu}$, $L+1\leq \mu\leq K$ 
\[
n_{\mu} \asymp J^{d-1} \, 2^J \, 2^{(L-\mu)\gamma} = J^{d-1} \, 2^{\mu}\, 2^{L-\mu}\, 2^{-(d-1)[\log J]}\, 2^{(L-\mu)\gamma}=
2^{\mu} \, 2^{(L-\mu)(\gamma+1)} \, ,
\]
where we used the definition of $L$. Observe $n_\mu \le D_\mu/2$.
The inequality \eqref{case3} and property (a) in Appendix A lead to the estimate
\beqq
x_{n_{\mu}}(id_{\mu})&\lesssim & 2^{\mu(- t+\frac{1}{p_0}-\frac{1}{p})}\, x_n(id_{p_0,p}^{D_{\mu}})
 \lesssim 2^{\mu(- t+\frac{1}{p_0}-\frac{1}{p})}\, (2^{\mu}\, 2^{(L-\mu)(\gamma + 1)})^{\frac{1}{p}-\frac{1}{p_0}}
\\
& = & 2^{-\mu t}\, 2^{(L-\mu)(\gamma +1)(\frac{1}{p}-\frac{1}{p_0})}\, .
\eeqq
This implies
\[ 
\sum_{\mu=L+1}^{K} x^\rho_{n_{\mu}}(id_{\mu}^*) \lesssim \sum_{\mu=L+1}^{K} 2^{-\mu\rho t}\, 
2^{(L-\mu)(\gamma + 1)(\frac{1}{p}-\frac{1}{p_0})\rho}\, .
\]
Choosing $\gamma >0$ such that
$$
 t> (\gamma + 1)\, \Big(\dfrac{1}{p_0}-\dfrac{1}{p}\Big)
$$
we conclude
\[
\sum_{\mu=L+1}^{K} x^\rho_{n_{\mu}}(id_{\mu}^*) \lesssim 2^{-Lt\rho}\asymp 2^{-tJ\rho}J^{-(d-1)t\rho}\, .
\]
Hence, inserting the previous inequality and \eqref{ws-38} into \eqref{sum2},
$$x_n(id^*)\lesssim 2^{-tJ}J^{(d-1)(-t)}$$
follows. Based on this estimate and \eqref{ws-37} one can finish the proof as before.
\end{proof}

As a corollary of Lem. \ref{the5-1} - Lem. \ref{the5-4} we obtain the main result of this subsection.

\begin{theorem}\label{the5-5}
Let $0 < p_0, p \le 2$ and $t > \Big(\frac{1}{p_0}-\frac{1}{p}\Big)_+$.
\\
{\rm (i)}
If $t > \frac{1}{p_0}-\frac 12$, then
\[
x_n(id^*) \asymp n^{-t} (\log n)^{(d-1)(t+\frac{1}{2}-\frac{1}{p_0})}
\]
holds for all $n \ge 2$.
\\
{\rm (ii)}
If $ t < \frac{1}{p_0}-\frac 12$, then
\[
x_n(id^*) \asymp n^{-t} 
\]
holds for all $n \ge 1$.
\end{theorem}

\begin{remark}
\rm
Again we comment on the limiting situation $t=\frac{1}{p_0}-\frac{1}{2}$.
For $0<p_0,p<2$, $\rho :=\min (1,p)$ and $t=\frac{1}{p_0}-\frac{1}{2}$ it follows
\[
n^{-t}\lesssim x_n(id^*)\lesssim n^{-t}(\log\log n)^{t+\frac1\rho},\qquad n\geq 3.
\]
This is the only limiting case where the gap is of order $\log \log n$ to some power.
For that reason we give a few more details.
In principal we argue as in Lemma \ref{the5-3}.
For given $J \in \N$, $J \ge 4$,   we choose $ L:=J+(d-1)\, [\log J]$ as above. Next we define  
\[
n_{\mu}:= \Big[\frac{2^J \, J^{d-1}}{\log J}\Big] \, , \qquad J+1 \le \mu \le L\, .
\]
Then
\[
\sum_{\mu=J+1}^L n_{\mu} \asymp 2^J \, J^{d-1} 
\]
and 
\[
x_{n_{\mu}}(id_{p_0,2}^{D_{\mu}})\lesssim \Big(\frac{2^J \, J^{d-1}}{\log J}\Big)^{\frac{1}{2}-\frac{1}{p_0}}
\]
follow, see property (a) in Appendix A.
Applying \eqref{case1} we find
\beqq
\sum_{\mu=J+1}^{L} x^{\rho}_{n_{\mu}}(id_{\mu}^*)& \lesssim & 
\sum_{\mu=J+1}^{L}2^{\mu{\rho}(- t+\frac{1}{p_0}-\frac{1}{2})} \, \Big(\frac{2^J \, J^{d-1}}{\log J}\Big)^{\rho(\frac{1}{2}-\frac{1}{p_0})}
\\
& \lesssim & \big(2^J \, J^{d-1}\big)^{-\rho t} \, (\log J)^{1+ \rho t} \, .
\eeqq
As in Lemma \ref{the5-3} this proves the claim.
\end{remark}


\subsubsection{The case $0<p \leq 2 < p_0 \leq \infty$}


This is the last case we have to consider.

\begin{theorem}\label{the3-1}
Let $0< p \leq 2 < p_0\leq \infty$ and $ t>\frac{1}{p_0}$. Then
\[
x_n(id^*)\asymp n^{-t+\frac{1}{p_0}-\frac{1}{2}}(\log n)^{(d-1)(t+\frac{1}{2}-\frac{1}{p_0})}
\]
holds for all $n\ge 2$.
\end{theorem}

\begin{proof}
{\em Step 1.} Estimate from below. 
Since $p\leq 2$, from (\ref{eqlow1}) and (\ref{case1}) we derive
\[
x_n(id^*) \gtrsim \mu^{(d-1)(\frac{1}{2}-\frac{1}{p})} \, 2^{\mu(-t+\frac{1}{p_0}-\frac{1}{p})} \, x_n(id_{p_0,p}^{D_\mu})
\, .
\]
We choose $n:=[D_{\mu}/2]$ and obtain from 
property (c)(part(i)) in Appendix A that
\[
x_n(id_{p_0,p}^{D_{\mu}} )\gtrsim (D_{\mu})^{\frac{1}{p}-\frac{1}{2}}\gtrsim (\mu^{d-1}\, 2^{\mu})^{\frac{1}{p}-\frac{1}{2}}\, .
\]
This implies
\[ 
x_n(id^*)\gtrsim 2^{\mu(-t-\frac{1}{2}+\frac{1}{p_0})} \, .
\]
Using $2^{\mu} \asymp \frac{n}{\log^{d-1}n}$ we conclude
\[
x_n(id^*)\gtrsim n^{-t-\frac{1}{2}+\frac{1}{p_0}}(\log n)^{(d-1)(t+\frac{1}{2}-\frac{1}{p_0})}\, .
\]
{\em Step 2.} Estimate from above.
We consider the commutative diagram

\tikzset{node distance=4cm, auto}

\begin{center}
\begin{tikzpicture}
 \node (H) {$s^{t,\Omega}_{p_0,p_0}b$};
 \node (L) [right of =H] {$s^{0,\Omega}_{p,2}f $};
 \node (L2) [right of =H, below of =H, node distance = 2cm ] {$ s^{0,\Omega}_{2,2}f$};
 \draw[->] (H) to node {$id^*$} (L);
 \draw[->] (H) to node [swap] {$id^1$} (L2);
 \draw[->] (L2) to node [swap]{$id^2$} (L);
 \end{tikzpicture}
\end{center}

From $p< 2$ we derive $ s^{0,\Omega}_{2,2}f \hookrightarrow s^{0,\Omega}_{p,2}f$ which implies 
$\| \, id^2\, \|< \infty$. The ideal property of the $s$-numbers in combination with Thm. \ref{the4-1} yield
\[ 
x_n (id^*) \lesssim n^{-t+\frac{1}{p_0}-\frac{1}{2}} \, (\log n)^{(d-1)(t+\frac{1}{2}-\frac{1}{p_0})} 
\]
if $t>\frac{ {1}/{2}- {1}/{p_0}}{ {p_0}/{2}-1}= \frac{1}{p_0}$. 
\end{proof}

\begin{theorem}\label{the3-2}
Let $0< p\leq 2< p_0< \infty$ and $0< t<\frac{1}{p_0}$. Then
\[
 x_n(id^*)\asymp n^{-\frac{tp_0}{2}}(\log n)^{(d-1)(t+\frac{1}{2}-\frac{1}{p_0})}
\]
holds for all $n\ge 2$.
\end{theorem}

\begin{proof}
{\em Step 1.} Estimate from below. Since $p\leq 2$, from (\ref{eqlow1}) and (\ref{case1}) we obtain
\[
x_n(id^*)\gtrsim \mu^{(d-1)(\frac{1}{2}-\frac{1}{p})} \, 2^{\mu(-t+\frac{1}{p_0}-\frac{1}{p})} \, x_n(id_{p_0,p}^{D_\mu})
\, .
\]
With $n:=[D_{\mu}^{\frac{2}{p_0}}]$ 
property (c)(part(ii)) in Appendix A yields
\[
x_n(id_{p_0,p}^{D_{\mu}} )\gtrsim D_{\mu}^{\frac{1}{p}-\frac{1}{p_0}} \gtrsim (\mu^{d-1}\, 2^{\mu})^{\frac{1}{p}-\frac{1}{p_0}}
\, .
\]
Hence
\[ 
x_n(id^*) \gtrsim \mu^{(d-1)(\frac{1}{2}-\frac{1}{p_0})}\, 2^{-t\mu}\, .
\]
Since  $2^{\mu}\asymp \frac{n^{\frac{p_0}{2}}}{\log^{d-1}(n^{\frac{p_0}{2}})}$ 
we conclude
\[
x_n(id^*)\gtrsim n^{-\frac{tp_0}{2}}(\log n)^{(d-1)(t+\frac{1}{2}-\frac{1}{p_0})}\, .
\]
{\em Step 2.} Estimate from above. Again we consider the commutative diagram
\tikzset{node distance=4cm, auto}
\begin{center}
\begin{tikzpicture}
 \node (H) {$s^{t,\Omega}_{p_0,p_0}b$};
 \node (L) [right of =H] {$s^{0,\Omega}_{p,2}f $};
 \node (L2) [right of =H, below of =H, node distance = 2cm ] {$ s^{0,\Omega}_{2,2}f$};
 \draw[->] (H) to node {$id^*$} (L);
 \draw[->] (H) to node [swap] {$id^1$} (L2);
 \draw[->] (L2) to node [swap]{$id^2$} (L);
 \end{tikzpicture}
\end{center}
\noindent
In addition we know
\[
x_n(id^1)\asymp n^{-\frac{tp_0}{2}}(\log n)^{(d-1)(t+\frac{1}{2}-\frac{1}{p_0})}\, , \qquad n \geq 2\, .
\]
if $2 < p_0 < \infty$ and $0<t<\frac{1}{p_0}$, see  Thm. \ref{the4-2}. Now the ideal property of the $s$-numbers 
yields
\[
x_n(id^*) \lesssim n^{-\frac{tp_0}{2}}(\log n)^{(d-1)(t+\frac{1}{2}-\frac{1}{p_0})}\, , \qquad n \geq 2\,
\]
if $0<t<\frac{1}{p_0}$.
\end{proof}

\begin{remark}
\rm
In the limiting situation $t=\frac{1}{p_0}>0$
we have
\[
n^{-\frac{1}{2}}(\log n)^{\frac{(d-1)}{2}}\lesssim x_n(id^*)\lesssim n^{-\frac{1}{2}}(\log n)^{\frac{(d-1)}{2}}(\log n)^{\frac{1}{2}+\frac{1}{\rho}}
\]
for all $n \ge 2$. Here $\rho=\min(1,p)$.
\end{remark}


\section{Proofs}
\label{proof}


Here we will give proofs of the assertions in Section \ref{main1}.
For better readability we continue to work with  $(p_0,p)$ instead of $(p_1,p_2)$. 


\subsection{Proof of the main Theorem \ref{main}}


The heart of the matter is the following in principal well-known lemma.

\begin{lemma}\label{weyl}
Let $0 < p_0 \le \infty$, $0< p< \infty$ and $t\in \re$. Then
\[
x_n( id^*: s_{p_0,p_0}^{t,\Omega}b \to s_{p,2}^{0,\Omega}f)\asymp x_n\big(id: S_{p_0,p_0}^t B(\Omega)\to S_{p,2}^0 F(\Omega)\big)
\]
holds for all $n \in \N$.
\end{lemma}

\begin{proof}
{\em Step 1.} 
Let $0 < p_0 < \infty$.
Let $E: B^{t}_{p_0,p_0} (0,1) \to B^{t}_{p_0,p_0} (\re) $ denote a linear and continuous extension operator.
For existence of those operators we refer, e.g., to \cite[3.3.4]{Tr83} or \cite{Ry}. Without loss of generality we may assume that 
\[
\supp Ef \subset \bigcup_{{\bar{\nu} \in \N_0^d}, \, \,   \bar{m}\in A_{\bar{\nu}}^\Omega }  \supp \Psi_{\bar{\nu},\bar{m}}\, , 
\] 
see Section \ref{domino}, for all $f$.
Then the $d$-fold tensor product operator
\[
\ce_d := E \otimes \ldots \otimes E 
\]
maps the tensor product space $S^t_{p_0,p_0} B((0,1)^d) = B^{t}_{p_0,p_0} (0,1) \otimes_{\gamma_{p_0}} \ldots \otimes_{\gamma_{p_0}} B^{t}_{p_0,p_0} (0,1)  $ 
into the tensor product space 
$S^t_{p_0,p_0} B(\R) = B^{t}_{p_0,p_0} (\re) \otimes_{\gamma_{p_0}} \ldots \otimes_{\gamma_{p_0}} B^{t}_{p_0,p_0} (\re)$, see \cite{SUt}, 
and is again a linear and continuous extension operator.
This follows from the fact that $\gamma_{p_0}$ is an uniform quasi-norm.
Hence  $\ce_d \in \cl (S^{t}_{p_0,p_0}B (\Omega), S^{t}_{p_0,p_0}B (\R))$.
\\
{\em Step 2.} Let $p_0 = \infty$. We discussed extension operators in this case in Subsection \ref{extrem2}.
Now we can argue as in Step 1.
\\
{\em Step 3.} We follow \cite{Vybiral} and consider the commutative diagram
\[
\begin{CD}
S_{p_0,p_0}^{t}B(\Omega) @ >\mathcal{E}_d >> S_{p_0,p_0}^{t}B(\mathbb{R}^d) @>\mathcal{W}>>s_{p_0,p_0}^{t,\Omega}b\\
@V id VV @. @VV id^*V\\
S_{p,2}^{0}F(\Omega) @ <R_{\Omega}<< S_{p,2}^{0}F(\mathbb{R}^d)@ <\mathcal{W}^* << s_{p,2}^{0,\Omega}f
\end{CD}
\]
The mapping $\mathcal{W}$ is defined as 
\[ 
\mathcal{W} f := \, \Big( 2^{|\bar{\nu}|_1}\, \langle f, \, \Psi_{\bar{\nu}, \bar{k}}\rangle
\Big)_{\bar{\nu}\in \N_0^d, \, \bar{k} \in A^{\Omega}_{\bar{\nu}}}\, .
\]
Furthermore, $\mathcal{W}^*$ is defined as 
\[
\mathcal{W}^* \lambda := 
\sum_{\bar{\nu} \in \N_0^ d}
\sum_{\bar{k} \in A^{\Omega}_{\bar{\nu}}} \lambda_{\bar{\nu}, \bar{k}} \, \Psi_{\bar{\nu}, \bar{k}}
\]
and $R_\Omega$ means the restriction to $\Omega$.
The boundedness of $\ce_d, \mathcal{W}, \mathcal{W}^*, R_\Omega$ and the ideal property of the $s$-numbers  yield
$x_n (id)\lesssim x_n (id^*)$.
A similar argument with a slightly modified diagram yields $x_n (id^*)\lesssim x_n (id)$ as well.
\end{proof}
Next we need to recall an adapted Littlewood-Paley assertion, see Nikol'skij \cite[1.5.6]{Ni}.

\begin{lemma}\label{littlewood}
Let $1 < p< \infty$. Then
\[
S_{p_,2}^{0} F (\R) = L_p (\R) \qquad \mbox{and}\qquad S_{p_,2}^{0} F (\Omega) = L_p (\Omega)
\]
in the sense of equivalent norms.
\end{lemma}

\noindent
{\bf Proof of Thm. \ref{main}}.
Lemma \ref{weyl} and Lemma \ref{littlewood} allow to carry over
the results obtained in Section \ref{sequence} to the level of function spaces. 
Theorem \ref{main} becomes a consequence of Theorems \ref{the1-1} - \ref{the4-2}, Theorem \ref{the5-5} and Theorems \ref{the3-1}, \ref{the3-2}. 
\qed


\subsection{Proofs of the results in Subsections \ref{extrem}}
\label{proofextrem}


Recall that  $s^{t,\Omega}_{\infty,q}a$ or $(s^{t,\Omega}_{\infty,q}a)_\mu $  has to 
be interpreted as $s^{t,\Omega}_{\infty,q}b$ and $(s^{t,\Omega}_{\infty,q}b)_\mu $. 

\begin{lemma}\label{lift2}
Let $t,r\in \mathbb{R}$ and $0<p,q,p_0,q_0\leq \infty$. Then 
\[
x_n(id^1: s_{p_0,q_0}^{t, \Omega}a \to s_{p,q}^{r,\Omega}a )
 \asymp x_n(id^2: s_{p_0,q_0}^{t-r,\Omega}a\to  s_{p,q}^{0,\Omega}a )\, , \qquad n \in \N\, .
\]
\end{lemma}

\begin{proof}
We consider the commutative diagram
\[
\begin{CD}
s_{p_0,q_0}^{t, \Omega}a @ > id^1 >> s_{p,q}^{r,\Omega}a\\
@V J_{r} VV @AA J_{-r} A\\
s_{p_0,q_0}^{t-r,\Omega}a @ >id^2>> s_{p,q}^{0,\Omega}a.
\end{CD}
\]
Here $J_{r}$ is the isomorphism defined in \eqref{lift1}.
Hence $x_n (id^1) \lesssim x_n (id^2)$.
But
\[
\begin{CD}
s_{p_0,q_0}^{t-r, \Omega}a @ > id^2 >> s_{p,q}^{0,\Omega}a\\
@V J_{-r} VV @AA J_{r} A\\
s_{p_0,q_0}^{t,\Omega}a @ >id^1>> s_{p,q}^{r,\Omega}a
\end{CD}
\]
yields $x_n (id^2) \lesssim x_n (id^1)$ as well. The proof is complete.
\end{proof}
\noindent
{\bf Proof of Theorem \ref{satz01}}. 
{\em Step 1.} Estimate from above. Under the given restrictions there always exists some $r>\frac{1}{2}$ such that $t>r+\Big(\frac{1}{p_0}-\frac{1}{2}\Big)_+$.
We  consider the commutative diagram
\tikzset{node distance=4cm, auto}
\begin{center}
\begin{tikzpicture}
 \node (H) {$S^{t}_{p_0,p_0}B((0,1)^d)$};
 \node (L) [right of =H] {$L_\infty((0,1)^d)$};
 \node (L2) [right of =H, below of =H, node distance = 2cm ] {$S^{r}_{2,2}B((0,1)^d)$};
 \draw[->] (H) to node {$id_1$} (L);
 \draw[->] (H) to node [swap] {$id_2$} (L2);
 \draw[->] (L2) to node [swap] {$id_3$} (L);
 \end{tikzpicture}
\end{center}
The multiplicativity of the Weyl numbers yields 
\[
x_{2n-1} (id_1)\le x_n (id_2)\, x_n (id_3)\, .
\]
From Lemmas \ref{littlewood} and \ref{lift2} we have
\be\label{lift3}
x_n(id_2)\asymp x_n(id:S^{t-r}_{p_0,p_0}B((0,1)^d)\to L_2((0,1)^d) ).
\ee
Prop. \ref{satz 2}, Thm. \ref{main} and \eqref{lift3} lead to
\[
x_{2n-1} (id_1)\lesssim \frac{(\log n)^{(d-1)r}}{n^{r-\frac 12}}\, \left\{
\begin{array}{lll}
\frac{(\log n)^{(d-1)(t-r- \frac{1}{p_0}+\frac{1}{2} )}}{n^{t-r}} & \quad & \mbox{if}\quad 0 < p_0 \le 2\, , \: 
t-r> \frac{1}{p_0}-\frac{1}{2}\, , 
\\
&& \\
\frac{(\log n)^{(d-1)(t-r - \frac{1}{p_0}+\frac{1}{2} )}}{n^{t-r -\frac{1}{p_0}+\frac{1}{2}}} 
& \quad & \mbox{if}\quad 2 \le p_0 \le \infty\, , \: t-r> \frac{1}{p_0} \, .
\end{array}
\right.
\]
Finally, the monotonicity of the Weyl numbers yields the claim for all $n\geq 2.$
\\
\\
{\em Step 2.} Estimate from below.
The claim will follow from the next proposition.

\begin{proposition}\label{satz03}
Let $t>\frac{1}{p_0}$. As estimates from below we get 
\beqq
x_{n} (id:\ S^{t}_{p_0,p_0}B((0,1)^d) &\to& L_\infty((0,1)^d)) 
\\
 & \gtrsim & 
\left\{
\begin{array}{lll}
\frac{(\log n)^{(d-1)(t+ \frac 12 -\frac{1}{p_0} )}}{n^{t-\frac{1}{2}}} & \quad & \mbox{if}\quad 0 < p_0 \le 2\, , 
\\
&& \\
\frac{(\log n)^{(d-1)(t + \frac 12 - \frac{1}{p_0} )}}{n^{t -\frac{1}{p_0}}} 
& \quad & \mbox{if}\quad 2 \le p_0 \le \infty\, , 
\nonumber
\end{array}
\right.
\eeqq
for all $ n\geq 2$.
\end{proposition}

\begin{proof}
Again we shall use the multiplicativity of the Weyl numbers, but this time in connection with its relation to 
the $2$-summing norm \cite[Lemma 8]{Pi2}. 
Let us recall this notion.\\
An operator $T\in \mathcal{L}(X,Y)$ is said to be \emph{absolutely $2$-summing} if there is a constant 
$C>0$ such that for 
all $n\in \mathbb{N}$ and $x_1 , \dots, x_n \in X$ the inequality
\begin{equation}\label{F2}
\big (\sum_{j=1} ^n \|\, T x_j \, |Y\|^2 \Big )^{1/2} \leq C 
\sup_{x^* \in X^*, \|x^*|X^*\|\leq 1} \Big ( \sum_{j=1} ^n |<x_j , x^* >|^2\Big)^{1/2}
\end{equation}
holds (see \cite[Chapter 17]{Pi1}). 
The norm $\pi _2 (T)$ is given by the infimum of all $C>0$ satisfying (\ref{F2}).
$X^*$ refers to the dual space of $X$.
Pietsch \cite{Pi2}  proved the inequality
\[
n^{1/2}\, x_n (S) \le \pi_2 (S)\, , \qquad n \in \N\, , 
\]
for any linear operator $S$.
Using this inequality with respect $S= id$  we conclude  
\beqq
&& \hspace{-0.7cm}
x_{2n-1} (id: ~ S^{t}_{p_0,p_0}B((0,1)^d) \to L_2((0,1)^d)) 
\\
& \le & 
x_{n} (id:~S^{t}_{p_0,p_0}B((0,1)^d) \to L_\infty ((0,1)^d)) \, 
x_{n} (id:~L_{\infty}((0,1)^d) \to L_2((0,1)^d)) 
 \\
& \le &
x_{n} (id:~S^{t}_{p_0,p_0}B((0,1)^d) \to L_\infty ((0,1)^d)) \, n^{-1/2}\, 
\pi_2 (id:~L_{\infty}((0,1)^d) \to L_2((0,1)^d)) 
\\
& = &
x_{n} (id:~S^{t}_{p_0,p_0}B((0,1)^d) \to L_\infty ((0,1)^d)) \, n^{-1/2}\, ;
\eeqq
where in the last equality we have used that
\[
\pi_{2} (id: L_\infty ((0,1)^d) \longrightarrow L_2 ((0,1)^d))=
\|id: L_\infty ((0,1)^d) \longrightarrow L_2 ((0,1)^d)\|=1\, , 
\]
see \cite[Example 1.3.9]{Pi3}). 
Since
\beqq
n^{\frac 12}\, x_{2n-1} (id: && \hspace{-0.7cm} S^{t}_{p_0,p_0}B((0,1)^d) \to L_2((0,1)^d))
\\
& \asymp &
\left\{
\begin{array}{lll}
\frac{(\log n)^{(d-1)(t+ \frac 12 -\frac{1}{p_0} )}}{n^{t-\frac{1}{2}}} & \quad & \mbox{if}\quad 0 < p_0 \le 2\, , \:t > \frac{1}{p_0} - \frac 12 \, , 
\\
&& \\
\frac{(\log n)^{(d-1)(t + \frac 12 - \frac{1}{p_0} )}}{n^{t -\frac{1}{p_0}}} 
& \quad & \mbox{if}\quad 2 \le p_0 \le \infty\, , \: t>\frac{1}{p_0} \, , 
\\
\end{array}
\right.
\eeqq
see Thm. \ref{the5-5}, Thm. \ref{the4-1}, this proves the claimed estimate from below.
\end{proof}


\subsection{Proof of the results in Subsection \ref{extrem1}}


As a preparation we need the following counterpart of the classical result $F^0_{1,2}(\R) \hookrightarrow L_1 (\R)$
in the dominating mixed situation. 
The following proof we learned from Dachun Yang and  Wen Yuan \cite{YY}.

\begin{lemma} We have
\[
S^0_{1,2}F(\mathbb{R}^d)\hookrightarrow L_1(\mathbb{R}^d).
\]
\end{lemma}

\begin{proof}
Let $f\in S^0_{1,2}F(\mathbb{R}^d)$. We may assume that $f$ is a Schwartz function, due to the density of $\cs (\R)$ 
in $S^0_{1,2}F(\mathbb{R}^d)$.
Let $(\varphi_{\bar{j}})_{\bar{j}}$ be the smooth dyadic decomposition of unity defined in 
\eqref{unityd}.
Let $\phi_0, \phi \in C_0^\infty (\re)$ be functions s.t.
\beqq
\phi_0 (t) & = &  1 \qquad \mbox{on}\quad \supp \varphi_0
\\
\phi (t) & = &  1 \qquad \mbox{on}\quad \supp \varphi_1\, .
\eeqq
We put 
$\phi_j (t):= \phi (2^{-j+1}t)$, $j \in \N$, and 
\[
\phi_{\bar{j}} := \phi_{{j_1}}\otimes \ldots \otimes
\phi_{{j_d}} \, , \qquad \bar{j}=(j_1, \ldots\, , j_d)\in
\N_0^d\, .
\]
It follows
\[
\sum_{\bar{j} \in \N_0^d} \varphi_{\bar{j}} (x)\, \cdot \, \phi_{\bar{j}} (x) = 1 \qquad \mbox{for all}\quad x \in \R\, ,
\]
see \eqref{unity} and \eqref{unityd}.
This implies 
\[ 
 f = \sum_{\bar{j}\in \mathbb{N}_0^d} \cfi [ \varphi_{\bar{j}} (\xi)\, \phi_{\bar{j}} (\xi)\, \cf f (\xi)] \qquad  \mbox{(convergence in $S'(\mathbb{R}^d)$)}.
\]
Let $g\in L_\infty(\mathbb{R}^d)$. H\"older's inequality yields
 \beqq
|\langle f,g\rangle | &\le& \sum_{\bar{j} \in \mathbb{N}_0^d} | \langle \cfi [\varphi_{\bar{j}} \cf f], \cfi [\phi_{\bar{j}} \, \cf g]\rangle| 
 \\
 & \le & 
\big\| \big(\sum_{\bar{j} \in \mathbb{N}_0^d } |\cfi [\varphi_{\bar{j}} \cf f]|^2\big)^{\frac{1}{2}} | L_1(\mathbb{R}^d) \big\|
\, \, 
 \big\| \big(\sum_{\bar{j} \in \mathbb{N}_0^d} |\cfi [\phi_{\bar{j}} \cf g]|^2\big)^{\frac{1}{2}} | L_\infty(\mathbb{R}^d)\big\|.
 \eeqq
Next we are going to use the tensor product structure of $\cfi \phi_{\bar{j}}$ and the fact that 
$\cfi \phi_{j_l}$, $l=1, \ldots \, , d$, are Schwartz functions.
For any $M>0$, we have
\beqq
|\cfi [\phi_{\bar{j}} \cf g](x)|& \lesssim & |\int_{\mathbb{R}^d} g(y)\, (\cfi \phi_{\bar{j}})(x-y)\, dy|
\\
 &\ls& \|g|L_{\infty}(\mathbb{R}^d)\| \, 
\int_{\mathbb{R}^n}\frac{ 2^{|\bar{j}|_1}}{\prod_{l=1}^{d}(1+2^{j_l}|x_l-y_l|)^{(1+M)}} dy 
 \\
&=& \|g|L_{\infty}(\mathbb{R}^d)\| \, \prod_{l=1}^{d} \, \int_{\mathbb{R}} \frac{ 2^{j_l}}{(1+2^{j_l}|x_l-y_l|)^{(1+M)}} \, dy \, .
\eeqq
Some elementary calculations yield
\[
 \int_{\mathbb{R}} \frac{ 2^{j_l}}{(1+2^{j_l}|x_l-y_l|)^{(1+M)}} \, dy \lesssim 
2^{-j_lM} 
\]
with constants independent of $j_l$. Inserting this in our previous estimate we obtain
\[
 |\cfi [\phi_{\bar{j}} \cf g](x)|\lesssim \|g|L_{\infty}(\mathbb{R}^d)\|\, 2^{-|\bar{j}|_1 M}. 
\]
Hence
\beqq
 \big\| \big(\sum_{\bar{j} \in \mathbb{N}_0^d} |\cfi [\phi_{\bar{j}} \cf g](x)|^2\big)^{\frac{1}{2}} | L_\infty\big\| 
&\ls& \|g|L_{\infty}(\mathbb{R}^d)\| \, \sum_{\bar{j} \in \mathbb{N}_0^d} \, 2^{-|\bar{j}|_1M} 
 \\
 &\ls& \|g|L_{\infty}(\mathbb{R}^d) \, \|\sum_{\mu=0}^{\infty}\, \sum_{|\bar{j}|_1=\mu}\, 2^{-|\bar{j}|_1 M}\, 
 \\
 &\ls& \|g|L_{\infty}(\mathbb{R}^d)\| \, .
 \eeqq
 Therefore, we obtain
\[
\|f| L_1(\mathbb{R}^d)\| = \sup_{\| g|L_\infty (\R)\|=1} \, |\langle f,g\rangle | \lesssim 
\| (\sum_{\bar{j}\in \mathbb{N}_0^d}|\cfi [\varphi_{\bar{j}} \cf f]|^2)^{\frac{1}{2}} | L_1(\mathbb{R}^d) \|.
\]
That completes our proof.
\end{proof}

\noindent {\bf Proof of Theorem \ref{extrem11}}. {\it Step 1.} Estimate from above. From the chain of embeddings
$$S_{p_0,p_0}^tB((0,1)^d)\hookrightarrow S^0_{1,2}F((0,1)^d)\hookrightarrow L_1((0,1)^d),$$
together with Lem. \ref{weyl}, Thm. \ref{the5-5}, Thm. \ref{the3-1}, Thm. \ref{the3-2} and the abstract properties of Weyl numbers, see Section \ref{basic}, we 
derive the upper bound. \\
{\it Step 2.} We prove the lower bound for the case $p_0<2 $ and $t< \frac{1}{p_0}-\frac{1}{2}$. 
First we note that, under the condition $t>\max(0,\frac{1}{p_0}-1)$, the chain of embeddings holds true
$$S_{p_0,p_0}^{t}B((0,1)^d) \hookrightarrow  L_1((0,1)^d) \hookrightarrow   S^0_{1,\infty}B((0,1)^d). $$
Then the ideal property of the $s$-numbers yields
\beq 
\label{ws-39}
x_n(id: S_{p_0,p_0}^{t}B((0,1)^d)\to S^0_{1,\infty}B((0,1)^d) )\lesssim x_n(S_{p_0,p_0}^{t}B((0,1)^d)\to L_1((0,1)^d)) \, .
\eeq
Next we consider the commutative diagram
\[
\begin{CD}
B^{t+r}_{p_0,p_0}(0,1) @ > id_1>> B^r_{1,\infty}(0,1)\\
@V \text{Ext} VV @AA \text{Tr} A\\
S_{p_0,p_0}^{t+r}B((0,1)^d) @ >id>> S^r_{1,\infty}B((0,1)^d)
\end{CD}
\]
Here the linear operators $\text{Ext}$ and $\text{Tr}$ are defined as follows. 
For $g\in B^{t+r}_{p_0,p_0}(0,1)$, we put
$$ (\text{Ext}g)(x_1,...,x_d)=g(x_1), \qquad x = (x_1,...,x_d) \in \R\, . $$
In case of $f\in S^r_{1,\infty}B((0,1)^d)$ we define
$$(\text{Tr} f)(x_1)=f(x_1,0,...,0)\, , \qquad x_1 \in \re \, .$$
Note that the condition $r>1$ guarantees that the operator $\text{Tr}$ is well defined, see \cite[Thm.~2.4.2]{ST}.
Furthermore, $\text{Ext}$ maps $B^{t+r}_{p_0,p_0}(0,1)$ continuously into  
$S_{p_0,p_0}^{t+r}B((0,1)^d)$. This follows from the fact that $\|\, \cdot \, |S_{p_0,p_0}^{t+r}B((0,1)^d)\|$ is a cross-quasi-norm, see the formula in Rem. \ref{blabla}(i).
Hence $id_1= \text{Tr} \circ id \circ \text{Ext} $ and 
\beq 
\label{ws-39-1}
x_n(id_1: && \hspace{-0.8cm} B^{t+r}_{p_0,p_0}(0,1)\to B^r_{1,\infty}(0,1) )
\\
& \lesssim & x_n(id: S_{p_0,p_0}^{t+r}B((0,1)^d)\to S^r_{1,\infty}B((0,1)^d) )\, .
\nonumber
\eeq
Making use of a lifting argument, see Lem. \ref{lift2}, we conclude that
\beq
\label{ws-39-2}
x_n(id: && \hspace{-0.8cm} S_{p_0,p_0}^{t+r}B((0,1)^d)\to S^r_{1,\infty}B((0,1)^d) )
\\
& \asymp & x_n(id: S_{p_0,p_0}^{t}B((0,1)^d)\to S^0_{1,\infty}B((0,1)^d) ) \, .
\nonumber
\eeq
The lower bound is now obtained from \eqref{ws-39}, \eqref{ws-39-1}, \eqref{ws-39-2} and 
\[
x_n(id_1:B^{t+r}_{p_0,p_0}(0,1)\to B^r_{1,\infty}(0,1) ) \asymp n^{-t},\qquad n \in \N\, , 
\]
if $0<p_0\leq 2$ and $t>\max(0,\frac{1}{p_0}-1)$, see Lubitz \cite{Claus} and Caetano \cite{Caed}.\\
{\it Step 3.} 
We prove  that 
$$x_n(id: S_{p_0,p_0}^tB((0,1)^d) \to  L_1((0,1)^d )) \gtrsim n^{-t}(\log n)^{(d-1)(t+\frac{1}{2}-\frac{1}{p_0})}$$
if $ p_0\leq 2$ and $ t> \frac{1}{p_0}-\frac{1}{2}$. There always exists 
a pair $(\theta,p)$ such that
\[
0<\theta<1\, , \qquad 1 < p < 2\, 
\]
and
\[
\| f|L_p((0,1)^d)\| \leq \|f|L_1((0,1)^d)\|^{1-\theta}\, \|f|L_{2}((0,1)^d)\|^{\theta}\qquad   \text{for all}\quad  f\in L_{2}((0,1)^d).
 \]
Next we employ the interpolation property of the Weyl numbers, see Thm. \ref{inter}, and obtain
\beqq
x_{2n-1}&& \hspace{-0.7cm}(id :  S^t_{p_0,p_0}B((0,1)^d)\to L_p((0,1)^d))\lesssim
\\  
&& x_n^{1-\theta}(id: S^t_{p_0,p_0}B((0,1)^d)\to L_1((0,1)^d))\, \, x_n^{\theta}(id:S^t_{p_0,p_0}B((0,1)^d)\to L_{2}((0,1)^d)).
\eeqq
Note that  $0< p_0\leq 2$ and $ t> \frac{1}{p_0}-\frac{1}{2}$ imply
\beqq
x_{n}(id :S^t_{p_0,p_0}B((0,1)^d)\to L_2((0,1)^d))&\asymp& x_{n}(id :S^t_{p_0,p_0}B((0,1)^d)\to L_{p}((0,1)^d))\\
&\asymp& n^{-t}(\log n)^{(d-1)(t+\frac{1}{2}-\frac{1}{p_0})},
\eeqq
see Thm. \ref{main}. This leads to
$$ x_n(id: S^t_{p_0,p_0}B((0,1)^d)\to L_1((0,1)^d))\gtrsim n^{-t}(\log n)^{(d-1)(t+\frac{1}{2}-\frac{1}{p_0})}.$$
The lower bounds in the remaining cases can be proved similarly.\qed


\subsection{Proofs of the results in Subsection \ref{extrem2} }
{\bf Proof of Theorem \ref{satz04}}. 
Define $id^*: s^{t, \Omega}_{p_0,p_0}b \to s^{0, \Omega}_{\infty,\infty}b$
and $id_\mu^*: (s^{t, \Omega}_{p_0,p_0}b)_\mu \to (s^{0, \Omega}_{\infty,\infty}b)_\mu$.
Cor. \ref{ba2-1} yields
\be\label{ws-56b}
\| \, id_\mu^* \, \|\lesssim 2^{\mu (\frac{1}{p_0}-t)}\, .
\ee
Arguing as in proof of Prop. \ref{wichtig1} one can establish the following.

\begin{lemma}\label{erg-4}
 Let $0< p_0 \le \infty$ and $t \in \re$.
Then 
\beqq
x_n(id_{\mu}^*) \asymp x_n (id_\mu^{**}: 2^{\mu(t-\frac{1}{p_0})} \ell_{p_0}^{D_\mu} \to \ell_{\infty}^{D_\mu} )
\asymp 2^{\mu(-t+\frac{1}{p_0})} \, x_n(id_{p_0,\infty}^{D_\mu})
\eeqq
for all $n \in \N$.
\end{lemma}

\noindent
Property (a) in Appendix A yields
\[
x_n(id_{p_0,\infty}^{D_\mu}) \asymp \left\{ \begin{array}{lll}
1 &\qquad & \mbox{if}\quad 2 \le p_0 \le \infty\, , \\
n^{\frac 12 - \frac{1}{p_0}} && \mbox{if} \quad 0< p_0 \le 2\, ,                      
\end{array}\right.
\]
if $2n \le D_\mu$.
Now we may follow the proof of Thm. \ref{the1-1}.
This results in the following useful statement.

\begin{theorem}\label{erg-1} 
{\rm (i)}
Let $0< p_0\leq 2$ and $t>\frac{1}{p_0}$. Then
\[
x_n(id^*: s^{t, \Omega}_{p_0,p_0}b \to s^{0, \Omega}_{\infty,\infty}b) 
\asymp n^{-t+\frac{1}{2}} (\log n)^{(d-1)(t-\frac{1}{p_0})} \, , \qquad n\geq 2\, .
\]
{\rm (ii)}
Let $2\leq p_0\leq \infty$ and $t>\frac{1}{p_0}$. Then
\[
x_n(id^*: s^{t, \Omega}_{p_0,p_0}b \to s^{0, \Omega}_{\infty,\infty}b) 
\asymp n^{-t+\frac{1}{p_0}} (\log n)^{(d-1)(t-\frac{1}{p_0})}\, , \qquad n\geq 2\, .
\]
\end{theorem}

By making use of a lifting argument, see Lemma \ref{lift2}, and the counterpart of Lemma \ref{weyl} for this situation, i.e.,
\[
x_n( id^*: s_{p_0,p_0}^{t,\Omega}b \to s_{\infty,\infty}^{0,\Omega}b) \asymp x_n
\big(id: S_{p_0,p_0}^t B(\Omega)\to S_{\infty,\infty}^0 B(\Omega)\big)\, , \qquad n \in \N\, , 
\]
we immediately get the following corollary.

\begin{corollary}\label{erg-2} 
Let $s \in \re$.\\
{\rm (i)}
Let $0< p_0\leq 2$ and $t> s+ \frac{1}{p_0}$. Then
\[
x_n(id: S^{t}_{p_0,p_0}B((0,1)^d) \to S^{s}_{\infty,\infty}B((0,1)^d)) \asymp n^{-t + s +\frac{1}{2}} (\log n)^{(d-1)(t-s-\frac{1}{p_0})} \, , \qquad n \geq 2\, .
\]
{\rm (ii)}
Let $2\leq p_0\leq \infty$ and $t>s+\frac{1}{p_0}$. Then
\[
x_n(id: S^{t}_{p_0,p_0}B((0,1)^d) \to S^{s}_{\infty,\infty}B((0,1)^d)) \asymp n^{-t+s +\frac{1}{p_0}} 
(\log n)^{(d-1)(t-s-\frac{1}{p_0})}\, , \qquad n \geq 2\, .
\]
\end{corollary}
\noindent
Now  Thm. \ref{satz04} follows from $\cz^{s}_{\mix}((0,1)^d)=S^{s}_{\infty,\infty}B((0,1)^d) $,
see Lemma \ref{zyg}.
\qed\\

\noindent
{\bf Proof of Theorem \ref{ap-cor}}.
The lower estimate in the case of high smoothness is a direct consequence of  $x_n \leq a_n$
and Theorem  \ref{satz04}.
\\
{\em Step 1.} We prove the upper bound of $a_{n} (id: S^{t}_{p_0,p_0}B((0,1)^d) \to S^{0}_{\infty,\infty}B((0,1)^d))$ in case $p_0 >1$. 
First, recall
\[
a_n(id_{p_0,\infty}^{D_\mu}) \asymp \left\{ \begin{array}{lll}
1 &\qquad & \mbox{if}\quad 2 \le p_0 \le \infty\, , \\
\min(1,D_{\mu}^{1-\frac{1}{p_0}}n^{-\frac{1}{2}}) && \mbox{if} \quad 1< p_0 < 2\, ,                     
\end{array}\right.
\]
if $2n \le D_\mu$, see \cite{Glus,Vy}.
To avoid nasty calculations by checking this behaviour for $p_0 \ge 2$
one may use the elementary chain of inequalities
\[
x_n(id_{p_0,\infty}^{D_\mu}) \leq a_n(id_{p_0,\infty}^{D_\mu}) \lesssim 1
\]
in combination with property (a) in Appendix A.\\
Let  $2\leq p_0\leq \infty$.
Because of $a_n(id_{p_0,\infty}^{D_\mu}) \asymp x_n(id_{p_0,\infty}^{D_\mu}) \asymp 1$ if $2n \le D_\mu$ 
we may argue as in case of Weyl numbers, see the proof of Thm. \ref{satz04} given above.
\\
Now we consider the case $ 1< p_0 < 2$ and $t>1$. We define
$$id^*: s^{t, \Omega}_{p_0,p_0}b \to s^{0, \Omega}_{\infty,\infty}b \quad 
\text{and}  \quad id_\mu^*: (s^{t, \Omega}_{p_0,p_0}b)_\mu \to (s^{0, \Omega}_{\infty,\infty}b)_\mu \, .$$
\eqref{ws-56b} and Lemma \ref{erg-4} yield
\[
\| \, id_\mu^* \, \|\lesssim 2^{\mu (\frac{1}{p_0}-t)}\, 
\]
and
\begin{equation}\label{ws-51a}
a_n(id_{\mu}^*) \asymp a_n (id_\mu^{**}: 2^{\mu(t-\frac{1}{p_0})} \ell_{p_0}^{D_\mu} \to \ell_{\infty}^{D_\mu} )
\asymp 2^{\mu(-t+\frac{1}{p_0})} \, a_n(id_{p_0,\infty}^{D_\mu})
\end{equation}
for all $n \in \N$.
Now we get as in Subsection 5.2, formula \eqref{sum1}, 
\begin{equation}\label{sum1b}
a_n(id^*)\lesssim \sum_{\mu=J+1}^{L} \, a_{n_{\mu}}(id_{\mu}^*)+ 2^{L (-t+\frac{1}{p_0})}\, ,
\end{equation}
since $\rho=1$ here.
For 
\be\label{lambda0}
1 < \lambda < \frac 12 + \frac t2 
\ee
we define 
$$ n_{\mu} := D_{\mu} \, 2^{(J-\mu)\lambda}\, , \qquad  J+1\leq \mu\leq L \, .$$
Then, as above, 
$$n_{\mu}\leq \frac{D_{\mu}}{2} \qquad \text{and}\qquad  \sum_{\mu=J+1}^{L}n_{\mu}\asymp J^{d-1}2^J$$
follows.
From (\ref{ws-51a}) and 
$ a_n(id_{p_0,\infty}^{D_\mu}) \asymp \min(1,D_{\mu}^{1-\frac{1}{p_0}}n^{-\frac{1}{2}})$ we conclude
\beqq
a_{n_{\mu}}(id_{\mu}^*)&\lesssim& 2^{\mu(-t+\frac{1}{p_0})} \, a_{n_\mu}(id_{p_0,\infty}^{D_\mu})\\
&\lesssim& 2^{\mu(-t+\frac{1}{p_0})}D_{\mu}^{1-\frac{1}{p_0}}  (D_{\mu}.2^{(J-\mu)\lambda})^{-\frac{1}{2}} \\
&\lesssim& 2^{\mu(-t+\frac{1}{2})}\mu^{(d-1)(\frac{1}{2}-\frac{1}{p_0})}2^{-\frac{1}{2}(J-\mu)\lambda}
\, .\eeqq
This leads to
\beqq
 \sum_{\mu=J+1}^{L} \, a_{n_{\mu}}(id_{\mu}^*)&\lesssim& \sum_{\mu=J+1}^{L}2^{\mu(-t+\frac{1}{2})}\mu^{(d-1)(\frac{1}{2}-\frac{1}{p_0})}2^{-\frac{1}{2}(J-\mu)\lambda}\\
 &\lesssim& 2^{J(-t+\frac{1}{2})}J^{(d-1)(\frac{1}{2}-\frac{1}{p_0})},
\eeqq
since $\lambda $ satisfies $\lambda < \frac 12 + \frac t2$, see \eqref{lambda0},
guaranteeing  the convergence of the series in that way. 
Now we choose $L$ large enough such that 
$$ 2^{L (-t+\frac{1}{p_0})} \lesssim 2^{J(-t+\frac{1}{2})}J^{(d-1)(\frac{1}{2}-\frac{1}{p_0})}.$$
In view of \eqref{sum1b} this yields
$$a_n(id^*)\lesssim 2^{J(-t+\frac{1}{2})}J^{(d-1)(\frac{1}{2}-\frac{1}{p_0})}.$$
This proves the estimate from above.
\\
{\em Step 2.} Let $p_0=1$. Then we use 
\[
a_n(id_{1,\infty}^{D_\mu}) \lesssim  n^{-\frac{1}{2}}                    
\]
if $2n \le D_\mu$, see \cite{Vy}.
This is just the limiting case of Step 1.
So we argue as there. \\
{\em Step 3.} It remains to consider the following case: $1<p_0<2$, $s=0$ and $\frac{1}{p_0}<t<1$. 
\\
{\em Substep 3.1.} Estimate from above.
In this case we define
\be\label{ws000}
L:=\Big[J\frac{p_0'}{2}+(d-1)p_0'\Big(\frac{1}{p_0}-\frac{1}{2}\Big)\log J\Big]
\ee
and
\[
n_{\mu}:= \big[D_{\mu}\, 2^{(\mu-L)\beta+J-\mu}\big] \, , \qquad J+1 \le \mu \le L\, , 
\]
for some $\beta>0$. Here $p_0'$ is the conjugate of $p_0$, i.e., $\frac{1}{p_0}+\frac{1}{p_0'}=1$. Again we have
$$n_{\mu}\leq \frac{D_{\mu}}{2} \qquad \text{and}\qquad  \sum_{\mu=J+1}^{L}n_{\mu}\asymp J^{d-1}2^J.$$
From (\ref{ws-51a}), (\ref{sum1b}) and
$ a_n(id_{p_0,\infty}^{D_\mu}) \asymp \min(1,D_{\mu}^{1-\frac{1}{p_0}}n^{-\frac{1}{2}})$ we get
\beqq
a_n(id^*)&\lesssim& \sum_{\mu=J+1}^{L} 2^{\mu(-t+\frac{1}{p_0})}D_{\mu}^{1-\frac{1}{p_0}}  \big[D_{\mu}\, 2^{(\mu-L)\beta+J-\mu}\big]^{-\frac{1}{2}}+ 2^{L (-t+\frac{1}{p_0})}\, \\
&\lesssim& \sum_{\mu=J+1}^{L} 2^{\mu\big(-t+1-\frac{\beta}{2}\big)}2^{\frac{L\beta-J}{2}}\mu^{(d-1)\big(\frac{1}{2}-\frac{1}{p_0}\big)}+ 2^{L (-t+\frac{1}{p_0})}\, .
\eeqq
The condition $t<1$ guarantees that we can choose $\beta>0$ such that $-t+1-\frac{\beta}{2}>0$. Then we have
\beqq
a_n(id^*)
&\lesssim&  2^{L\big(-t+1-\frac{\beta}{2}\big)}2^{\frac{L\beta-J}{2}}J^{(d-1)\big(\frac{1}{2}-\frac{1}{p_0}\big)}+ 2^{L (-t+\frac{1}{p_0})}\, \\
&=&  2^{L(-t+1)}2^{-\frac{J}{2}}J^{(d-1)\big(\frac{1}{2}-\frac{1}{p_0}\big)}+ 2^{L (-t+\frac{1}{p_0})}\, .
\eeqq
Now, replacing $L$ by the value in \eqref{ws000}, a simple calculation yields
\beqq
a_n(id^*) &\lesssim&  2^{J\frac{p_0'}{2}\big(-t+\frac{1}{p_0}\big)}J^{(d-1)\big[\frac{p_0'}{2}\big(-t+\frac{1}{p_0}\big) +t-\frac{1}{p_0}\big]}\, .
\eeqq
Rewriting this in dependence on $n$ we obtain
\beqq
a_n(id^*)\lesssim n^{-\frac{p_0'}{2}(t-\frac{1}{p_0})}(\log n)^{(d-1)(t- \frac{1}{p_0} )}.
\eeqq
This proves the estimate from above. 
\\
{\em Substep 3.2.} Estimate from below. First of all, notice that we can prove 
$ a_n(id^*)\geq a_n(id_{\mu}^*)$ as in  (\ref{eqlow1}). 
We choose $n=[D_{\mu}^{\frac{2}{p_0'}}]$.
By employing again (\ref{ws-51a}) and 
\[ 
a_n(id_{p_0,\infty}^{D_\mu}) \asymp \min(1,D_{\mu}^{1-\frac{1}{p_0}}n^{-\frac{1}{2}})
\] 
we obtain the desired estimate. 
Finally,  by making use of a lifting argument, see Lemma \ref{lift2}, and the counterpart of Lemma \ref{weyl} we finish our proof.\qed


\subsection{Proof of interpolation properties of Weyl numbers}


For the basics in interpolation theory we refer to 
the monographs \cite{BL,lun,t78}. 
\\
To begin with we deal with Gelfand numbers.
The  $n$-th Gelfand number is defined as 
\be\label{def-gel}
c_n (T)=\inf_{M_n}\sup_{\|x|X\|\leq 1, x\in M_n}\|Tx|Y\|
\ee
where $M_n$ is a subspace of $X$ such that $\text{codim}M_n< n$, see also Section \ref{basic}. 
Next we recall the interpolation properties of Gelfand numbers, for the case of  Banach spaces we refer to Triebel \cite{Tr70}.

\begin{theorem}\label{inte2}
Let $0<\theta<1$. 
Let $X, Y, X_0,Y_0$ be  quasi-Banach spaces. 
Further we assume $Y_0 \cap Y_1  \hookrightarrow Y$ and the existence of a positive constant $C$ with
\be\label{interq}
\| y|Y\| \leq C \, \| y|Y_0\|^{1-\theta}\, \|y|Y_1\|^{\theta} \qquad   \text{for all}\quad  y\in Y_0\cap Y_1.
 \ee
Then, if 
\[
T\in \mathcal{L}(X,Y_0) \cap \mathcal{L}(X,Y_1) \cap \cl(X,Y)
\]
it follows
\[
c_{n+m-1}(T:X\to Y)\le C\,   c_n^{1-\theta}(T: X\to Y_0)\, c_m^{\theta}(T:X\to Y_1)
\]
for all $n,m \in \N$. Here $C$ is the same constant as in \eqref{interq}.
\end{theorem}

\begin{proof}
We follow the proof in \cite{Tr70}. 
Let $L_n$ and $L_m$ be subspaces of $X$ such that $\text{codim}L_n< n$ and $\text{codim}L_m< m$ respectively. Then $\text{codim}(L_n\cap L_m)<m+n-1$. 
Furthermore, by assumption, for all $x \in X$ we have $Tx \in Y_0 \cap Y_1.$
From \eqref{def-gel} and \eqref{interq} we derive
\beqq
c_{m+n-1} (T:X\to Y)&=&\inf_{L_n,L_m}\sup_{\substack{\|x|X\|\leq 1\\ x\in L_n\cap L_m}}\|Tx|Y\|\\
&\leq &C\inf_{L_n,L_m}\sup_{\substack{\|x|X\|\leq 1\\ x\in L_n\cap L_m}}\| \, Tx\, |Y_0\|^{1-\theta}\, \|\, Tx\, |Y_1\|^{\theta}\\
&\leq &C\big(\inf_{L_n}\sup_{\substack{\|x|X\|\leq 1\\ x\in L_n}}\| \, Tx \, |Y_0\|\big)^{1-\theta}\big(\inf_{L_m}\sup_{\substack{\|x|X\|\leq 1\\ x\in L_m}}\| \, Tx\, |Y_1\|\big)^{\theta}\\
&=& C \,  c_n^{1-\theta}(T: X\to Y_0)\, c_m^{\theta}(T:X\to Y_1) \, .
\eeqq
The proof is complete.
\end{proof}

\begin{remark}
 \rm
Triebel \cite{Tr70} worked with Gelfand widths. For compact operators 
Gelfand widths and Gelfand numbers coincide, see also \cite{Tr70}.
Hence, if we require 
\[
T\in \mathcal{K}(X,Y_0) \cap \mathcal{K}(X,Y_1) \cap \cl(X,Y)\, , 
\]
where $\mathcal{K}(X,Y)$ stands for the subspace of $\mathcal{L}(X,Y)$ formed by the compact operators,
then Theorem \ref{inte2} remains true for Gelfand widths.
Without extra conditions on $T$ Gelfand widths and Gelfand numbers 
may not coincide, see Edmunds and Lang \cite{EL} for a discussion of this question.
\end{remark}

\noindent
Now we ready prove the Thm. \ref{inter}.\\
\noindent
{\bf Proof of Theorem \ref{inter}.} 
Let $A\in \mathcal{L}(\ell_2,X)$ such that $\|A\|\leq 1$. Then from Thm. \ref{inte2} we conclude
\[
c_{n+m-1}(TA:\ell_2\to Y)\le C\,   c_n^{1-\theta}(TA: \ell_2\to Y_0)\, c_m^{\theta}(TA:\ell_2\to Y_1).
\]
Employing Remark \ref{weylextra}(ii) we obtain
\[
c_{n+m-1}(TA:\ell_2\to Y)\le C\,   x_n^{1-\theta}(T: X\to Y_0)\, x_m^{\theta}(T: X\to Y_1).
\]
Now taking the supremum with respect to $A$ we find
$$x_{n+m-1}(T:X\to Y)\le C\,   x_n^{1-\theta}(T: X\to Y_0)\, x_m^{\theta}(T:X\to Y_1).$$
The proof is complete.
\qed


\section{Appendix A - Weyl numbers of the embeddings $\ell_{p_0}^m \to \ell_{p}^m$}
\label{appendixb}


The Weyl numbers of $id: \ell_{p_0}^m\to \ell_{p}^m$ have been investigated at various places, we refer to 
Lubitz \cite{Claus}, K\"onig \cite{Koe}, Caetano \cite{Cae1,Cae2} and Zhang, Fang, Huang \cite{Fang1}.
We shall need the following.

\begin{itemize}
 \item[(a)] (\cite[Korollar 2.2]{Claus} and \cite{Fang1}) Let $n,m\in \N$ and $2n\leq m$. Then we have
\[ 
x_n(id_{p_0,p}^{m})\asymp \left\{\begin{array}{lll}
1 & \qquad & \mbox{if}\quad 2 \le p_0 \le p\le \infty\, ,
\\
n^{\frac{1}{p}-\frac{1}{p_0}} & & \mbox{if} \quad 0 < p_0\leq p\leq 2\, , 
\\
n^{\frac{1}{2}-\frac{1}{p_0}} && \mbox{if} \quad 0< p_0 \leq 2 \leq p \leq \infty\,, 
\\
 m^{\frac{1}{p}-\frac{1}{p_0}} && \mbox{if} \quad 0< p<p_0\leq 2\, .
\end{array}\right.
\]
\item[(b)] (\cite[Korollare 2.6, 2.8, Satz 2.9]{Claus}) Let $2\leq p<p_0\leq \infty$ and $n,m,k\in\mathbb{N},\ k\geq 2$. Then we have
\begin{enumerate}
\item $x_n(id_{p_0,p}^{m})\lesssim \bigg(\dfrac{m}{n}\bigg)^{\frac{1}{r}}$ if $n\leq m$, $\ \dfrac{1}{r}=\dfrac{ {1}/{p}- {1}/{p_0}}{1- {2}/{p_0}}$,
\item $x_n(id_{p_0,p}^{m})\gtrsim m^{\frac{1}{p}-\frac{1}{p_0}}$ if $1\leq n\leq [m^{\frac{2}{p_0}}]$,
\item $x_n(id_{p_0,p}^{kn})\asymp 1$.
\end{enumerate}
\item[(c)] (\cite{Fang1})
Let $0< p\leq 2< p_0\leq \infty$ and $n, m\in \N$. Then
\begin{enumerate}
\item $x_n(id_{p_0,p}^m)\gtrsim m^{\frac{1}{p}-\frac{1}{2}}$ if $n\leq \frac{m}{2}$,
\item $x_n(id_{p_0,p}^m)\gtrsim m^{\frac{1}{p}-\frac{1}{p_0}}$ if $n\leq m^{\frac{2}{p_0}}$.
\end{enumerate}

\end{itemize}


\section{Appendix B - Function spaces of dominating mixed smoothness}
\label{appendixa}



\subsection{Besov and Lizorkin-Triebel spaces on $\re$}
\label{space1}


Here we recall the definition and a few properties of Besov and
Sobolev spaces defined on $\re$. We shall use the Fourier analytic
approach, see e.g. \cite{Tr83}. Let $\varphi \in C_0^\infty (\re)$
be a function such that $\varphi (t)=1$ in an open set containing
the origin. Then by means of
\be\label{unity}
\varphi_0 (t) = \varphi (t)\, , \qquad \varphi_j(t)=
\varphi(2^{-j}t)-\varphi(2^{-j+1}t)\, , \qquad t\in \re\, , \quad j
\in \N\, ,
\ee
we get a smooth dyadic decomposition of unity, i.e.,
\[
\sum_{j=0}^\infty \varphi_j(t)= 1 \qquad \mbox{for all}\quad t \in \re\, ,
\]
and $\supp \varphi_j$ is contained in the dyadic annulus $\{t\in \re: \quad a \, 2^j \le |t| \le b \, 2^j \}$
with $0 < a < b < \infty$ independent of $j \in \N$.

\begin{definition} Let $0< p,q \le\infty$ and $s \in \re$.
\\
{\rm (i)}
The Besov space $B^s_{p,q}(\re)$ is then the collection of all
tempered distributions $f\in \mathcal{S}'(\re)$ such that
\[
\|\, f \, |B^s_{p,q}(\re)\|:= \Big(\sum\limits_{j=0}^{\infty}
2^{jsq}\, \|\, \cfi [\varphi_j \cf f] (\, \cdot\, )\,
|L_p(\re)\|^q\Big)^{1/q}
\]
is finite. 
\\
{\rm (ii)} Let $p< \infty$.
The Lizorkin-Triebel space $F^s_{p,q}(\re)$ is then the collection of all
tempered distributions $f\in \mathcal{S}'(\re)$ such that
\[
\|\, f \, |F^s_{p,q}(\re)\|:= \Big\| \Big(\sum\limits_{j=0}^{\infty}
2^{jsq}\, |\, \cfi [\varphi_j \cf f] (\, \cdot\, )\,|^q\Big)^{1/q}\, \Big| L_p(\re)\Big\|
\]
is finite. 
\end{definition}

\begin{remark}
 \rm
{\rm (i)} There is an extensive literature about Besov and Lizorkin-Triebel spaces, we refer to the monographs 
 \cite{Ni}, \cite{Tr83}, \cite{Tr92} and \cite{Tr06}.
These quasi-Banach spaces $B^s_{p,q}(\re)$ and $F^s_{p,q}
(\re)$ can be characterized in various ways, e.g. by differences and
derivatives, whenever $s$ is sufficiently large, i.e., $s > \max(0,1/p-1)$ in case of Besov spaces and
$s > \max(0,1/p-1, 1/q-1)$ in case of Lizorkin-Triebel spaces. 
We refer to \cite{Tr83} for details. \\ 
{\rm (ii)} The spaces $B^s_{p,q}(\re)$ and $F^s_{p,q} (\re)$ do not
coincide as sets except the case $p=q$. 
 \end{remark}


\subsection{Besov and Lizorkin-Triebel spaces of dominating mixed smoothness}
\label{space2}


Detailed treatments of Besov and Lizorkin-Triebel spaces of dominating
mixed smoothness are given at various places, we refer to the
monographs \cite{Am,ST}, the survey \cite{Sc2} as well as to the booklet \cite{Vybiral}.\\
If $\varphi_j$, $j \in \N_0$, is a smooth dyadic decomposition of
unity as in \eqref{unity}, then by means of
\be\label{unityd}
\varphi_{\bar{j}} := \varphi_{{j_1}}\otimes \ldots \otimes
\varphi_{{j_d}} \, , \qquad \bar{j}=(j_1, \ldots\, , j_d)\in
\N_0^d\, ,
\ee
we obtain a smooth decomposition of unity on $\R$.

\begin{definition} Let $0<p,q \le \infty$ and $t \in\re$.
\\
{\rm (i)}
The Besov space of dominating mixed smoothness $ S^{t}_{p,q}B(\re^d)$ is the
collection of all tempered distributions $f \in \mathcal{S}'(\R)$
such that
\[
 \|\, f \, |S^{t}_{p,q}B(\R)\| :=
\Big(\sum\limits_{\bar{j}\in \N_0^d} 2^{|\bar{j}|_1 t q}\, \|\, \cfi[\varphi_{\bar{j}}\, \cf f](\, \cdot \, )
|L_p(\re^d)\|^q\Big)^{1/q}
\]
is finite. 
\\
{\rm (ii)} Let $0 < p< \infty$.
The Lizorkin-Triebel space of dominating mixed smoothness $ S^{t}_{p,q}F(\re^d)$ is the
collection of all tempered distributions $f \in \mathcal{S}'(\R)$
such that
\[
 \|\, f \, |S^{t}_{p,q}F(\R)\| :=
\Big\| \Big(\sum\limits_{\bar{j}\in \N_0^d} 2^{|\bar{j}|_1 t q}\, |\, \cfi[\varphi_{\bar{j}}\, \cf f](\, \cdot \, )|^q \Big)^{1/q} \Big|L_p(\re^d)\Big\|
\]
is finite. 
\end{definition}

\begin{remark}\label{blabla}
\rm
{\rm (i)} The most interesting property of these classes for us consists in the following: 
if 
\[
 f(x) = \prod_{j=1}^d f_j (x_j)\, , \qquad x=(x_1, \ldots \, , x_d)\, , \quad f_j \in A^t_{p,q}(\re)\, , \quad j=1, \ldots \, , d \, , 
\]
then $f \in S^t_{p,q}A(\R)$ and 
\[
 \| \, f \, | S^t_{p,q}A(\R)\| = \prod_{j=1}^d \|\, f_j \, |A^s_{p,q} (\re)\| \, , \qquad A \in \{B,F\}\, .  
\]
I.e., Lizorkin-Triebel and Besov spaces of dominating mixed smoothness have a cross-quasi-norm.\\ 
{\rm (ii)} These classes $S^{t}_{p,q}B(\R)$ as well
as $S^{t}_{p,q}F(\R)$ are quasi-Banach spaces. If either 
$t > \max (0, (1/p)-1)$ (B-case) or $t > \max (0, 1/p-1, 1/q-1)$ (F-case), then they can be
characterized by differences, we refer to \cite{ST} and \cite{U1} for details. \\
{\rm (iii)} Again the spaces $S^{t}_{p,q}B(\R)$ and
$S^{t}_{p,q}F(\R)$ do not coincide as sets except the
case $p=q$. 
\\
{\rm (iv)} For $d=1$ we have
\[
 S^{t}_{p,q} A(\re) = A^t_{p,q}(\re)\, , \qquad A \in \{B,F\}\, . 
\]
\end{remark}

{\bf Acknowledgement:} The authors would like to thank A. Hinrichs for a hint 
concerning a misprint in the phd-thesis of Lubitz \cite{Claus}, 
V.N. Temlyakov for a hint 
concerning a misprint in his paper  \cite{Tem2}, T. K\"uhn for an explanation how to use 
\eqref{kuehn} and Dachun Yang and Wen Yuan for a nice new proof of the continuous embedding $F^0_{1,2} (\R) \hookrightarrow L_1 (\R)$.


\end{document}